\documentclass[a4paper,fleqn]{cas-dc}

\usepackage[USenglish]{babel}

\def\tsc#1{\csdef{#1}{\textsc{\lowercase{#1}}\xspace}}
\tsc{WGM}
\tsc{QE}

\usepackage[short,nocomma]{optidef} 
\usepackage[square, numbers]{natbib}
\usepackage{amssymb}
\usepackage{amsmath}
\usepackage{subcaption}  
\usepackage{rotating}  
\usepackage{refcount}  
\usepackage{underscore}  

\usepackage{amsthm}
\newtheorem{myprop}{Proposition}
\newtheorem*{myprop*}{Proposition}
\newtheorem{mytheorem}{Theorem}
\newtheorem{mydef}[myprop]{Definition}

\usepackage{array}
\usepackage{booktabs}
\usepackage{colortbl}

\usepackage{xcolor,soul}
\definecolor{dr}{rgb}{0.75,0.00,0.00}
\definecolor{lr}{rgb}{1.00,0.75,0.75}
\sethlcolor{lr}

\usepackage{dcolumn}
\newcolumntype{d}[1]{D{.}{.}{#1}}
\newcommand{\dsh}{\multicolumn{1}{c}{\textemdash}}

\begin{document}
\let\WriteBookmarks\relax
\def\floatpagepagefraction{1}
\def\textpagefraction{.001}

\shorttitle{Resource-robust valid inequalities for vehicle routing and related problems}
\shortauthors{Y.N. Hoogendoorn \& K. Dalmeijer}

\title [mode = title]{Resource-robust valid inequalities for vehicle routing and related problems}     

\author[1]{Ymro N. Hoogendoorn}[type=editor,
                        auid=000,bioid=1,
                        orcid=0000-0002-9816-8842]
\cormark[1]
\ead{hoogendoorn@rsm.nl}

\affiliation[1]{organization={Erasmus University Rotterdam},
                addressline={Burgermeester Oudlaan 50}, 
                city={Rotterdam},
                postcode={3062PA}, 
                state={Zuid-Holland},
                country={The Netherlands}}

\author[2]{Kevin Dalmeijer}[type=editor,
                        auid=000,bioid=1,
                        orcid=0000-0002-4304-7517]
                        
\affiliation[2]{organization={H. Milton Stewart School of Industrial and Systems Engineering, Georgia Institute of Technology},
                city={Atlanta},
                state={GA},
                country={United States}}

\cortext[cor1]{Corresponding author}

\begin{abstract}
\noindent Branch-price-and-cut algorithms play an important role in solving many vehicle routing problems (VRPs).
Adding valid inequalities in this framework can impact the pricing subproblem, for which the literature distinguishes between `robust’ and `non-robust’ cuts.
We define the `robust application’ of a cut in a specific context, making this distinction more precise.
Next, we define broader `resource-robust applications' that can be handled efficiently in the subproblem.
We then introduce new resource-robust valid inequalities and show computational benefits for the capacitated VRP.
\end{abstract}


\maketitle

\section{Introduction}
\label{sec:introduction}

Branch-price-and-cut algorithms play an important role in solving a wide variety of vehicle routing problems, crew scheduling problems, and related problems such as bin packing, vector packing, and generalized assignment \citep{CostaEtAl2019-ExactBranchPrice,DesaulniersEtAl1998-UnifiedFrameworkDeterministic,PessoaEtAl2020-GenericExactSolver}.
In these applications, the column-generation subproblem, or \emph{pricing problem}, is commonly modeled as a Shortest Path Problem with Resource Constraints (\mbox{SPPRC}) and solved with a \emph{labeling algorithm} \citep{IrnichDesaulniers2005-ShortestPathProblems}.
Valid inequalities, or \emph{cuts} are added to strengthen the linear programming relaxations.

When adding cuts, it is important to consider how the duals can be incorporated into the pricing problem, as this may have a major impact on the performance of the labeling algorithm.
For this purpose, the literature distinguishes \emph{robust valid inequalities} and \emph{non-robust valid inequalities} \citep{DeAragaoEtAl2003-IntegerProgramReformulation,CostaEtAl2019-ExactBranchPrice,FukasawaEtAl2006-RobustBranchCut,PessoaEtAl2008-RobustBranchCut}.
Although definitions vary between papers, the common intuition is that robust cuts can be easily incorporated and have only a limited impact on the pricing problem.
Handling non-robust cuts in the labeling algorithm typically requires additional resources or specialized dominance rules, which can severely impact computational performance.
Despite this drawback, non-robust valid inequalities can still speed up the overall algorithm if used with care \citep{JepsenEtAl2008-SubsetRowInequalities, PessoaEtAl2020-GenericExactSolver}, and mitigating the negative impact on the labeling algorithm is an active field of study \citep{FaldumEtAl2023-SubsetRowInequalities,PecinEtAl2017-NewEnhancementsExact,PecinEtAl2017-ImprovedBranchCut,PessoaEtAl2020-GenericExactSolver}.

While the common terminology of robust versus non-robust cuts is often sufficient, it fails to accurately describe cases where context is important.
For example, \citep{PessoaEtAl2008-RobustBranchCut} introduce extended capacity cuts and triangle clique cuts that can be considered robust, but only for vehicle routing problems that use an extended graph in which vertices are duplicated for every possible vehicle load.
The route load knapsack cuts introduced by \citep{LiguoriEtAl2023-NonrobustStrongKnapsack} provide an even more unique example.
While the new cuts are presented as non-robust, the authors also show that if the labeling algorithm already includes load resources, then the duals can be incorporated efficiently, almost as if the cuts were robust.

In this paper, we first introduce more precise language to describe robust valid inequalities.
Rather than stating that a cut is robust, this paper will define what it means for a cut to admit a \emph{robust application} in the context of a specific model.
This definition makes it clear that robustness is dependent on the problem formulation and captures the intuition behind robustness in the literature.

Second, the paper provides the new concept of a \emph{resource-robust application} of a cut in the context of a specific model and a specific set of resources that are present in the labeling algorithm.
The new definition encompasses the route load knapsack cuts and generalizes the observations by \citep{LiguoriEtAl2023-NonrobustStrongKnapsack} to show that cuts with a resource-robust application can be efficiently incorporated into the labeling algorithm.
Within the necessary context, we also refer to cuts with a resource-robust application as \emph{resource-robust valid inequalities}.

Third, the paper demonstrates the merit of resource-robust applications by defining several new classes of valid inequalities and demonstrating their effectiveness.
These new valid inequalities can be applied resource-robustly to the $ng$-route relaxation resources \citep{BaldacciEtAl2004-ExactAlgorithmCapacitated}.
The $ng$-route relaxation is commonly used in state-of-the-art algorithms and the corresponding resources are therefore readily available \citep{CostaEtAl2019-ExactBranchPrice}, making the new valid inequalities widely applicable.
As a proof of concept, we perform computational experiments on the capacitated vehicle routing problem and demonstrate a substantial benefit of using resource-robust capacity cuts over robust or non-robust capacity cuts.

This paper is structured as follows.
Section~\ref{sec:setting} defines the general setting considered in this paper.
Robust and resource-robust applications are formally introduced in Section~\ref{sec:resource_robust}, followed by several new classes of valid inequalities in Section~\ref{sec:applications}.
Section~\ref{sec:compexp} presents computational experiments for the capacitated vehicle routing problem.
Finally, Section~\ref{sec:conclusion} states our conclusions and some directions for future research.

\section{Setting}
\label{sec:setting}
This section introduces the setting that is considered in this paper.
This is a common setting that is quite generic and that is used extensively in state-of-the-art algorithms \citep{CostaEtAl2019-ExactBranchPrice,DesaulniersEtAl1998-UnifiedFrameworkDeterministic,PessoaEtAl2020-GenericExactSolver}.

\paragraph{Model}
Let $G=(V,A)$ be a directed graph with vertices $V = \{0,1,\hdots,n+1\}$ and arc set $A$.
Vertices $0$ and $n+1$ correspond to the starting depot and the ending depot, respectively, and the set $V' = \{1,2,\hdots,n\}$ corresponds to customers that have to be served.
Arcs $A$ are defined from depot $0$ to customers $V'$, between customers $V'$, and from customers $V'$ to depot $n+1$.
The cost of traversing an arc $(i,j) \in A$ is given by $c_{ij} \in \mathbb{R}$.
Let $P$ be a set of routes $p \in P$ in $G$ from $0$ to $n+1$, possibly adhering to many restrictions. The parameter $a_p^i$ is the number of times vertex $i\in V'$ is served by route $p \in P$. 
With slight abuse of notation we use $i \in p$ to indicate that vertex $i\in V$ is on path $p \in P$, and $(i,j)\in p$ to indicate that arc $(i,j)\in A$ is on path $p \in P$.
Summations over $p$ may repeat the same vertices and arcs multiple times if $p$ has cycles. For example, a summation of values $y_i$, $i\in V'$ along route $p$ will conveniently be expressed as $\sum_{i\in p}y_i$ instead of $\sum_{i\in V'}a^i_py_i$.

The goal is to select a set of routes that serve all customers at minimum cost.
For this purpose we define variables $x_p \in \{0,1\}$ to be one if and only if route $p \in P$ is selected.
The model can now be expressed as follows, including a generic valid inequality.

\begin{figure}[th]
\begin{mini!}
%
	{}%
%
	{\sum_{p\in P} \sum_{(i,j) \in p} c_{ij} x_p, \label{eq:spp:obj}}
%
	{\label{eq:spp}}%
%
	{}
%
%
	\addConstraint%
	{\sum_{p\in P} a_p^i x_p = 1,}%
	{\quad\forall i\in V', \quad}%
	{(\mu^i \in \mathbb{R}) \label{eq:spp:constr}}
%
%
	\addConstraint%
	{\sum_{p\in P} g_p x_p \ge h,}%
	{~}%
	{(\gamma \ge 0) \label{eq:spp:vige}}
	\addConstraint%
	{x_p \in \mathbb{B},}%
	{\quad \forall p\in P.}%
	{\label{eq:spp:binary}}
\end{mini!}
\end{figure}

The Objective~\eqref{eq:spp:obj} minimizes the cost of the selected routes, Constraints~\eqref{eq:spp:constr} ensure that all customers are served, and Constraint~\eqref{eq:spp:vige} is a generic inequality for demonstration purposes.
Constraints~\eqref{eq:spp:binary} enforce that route selections are binary.

\paragraph{Pricing Problem}
In the branch-price-and-cut context, the integrality constraints~\eqref{eq:spp:binary} are relaxed to $x_p \ge 0$, and the resulting linear program is solved for a subset of the paths in $P$.
This results in duals for Constraints~\eqref{eq:spp:constr}, \eqref{eq:spp:vige} indicated by $\mu^i \in \mathbb{R}$ $\forall i \in V'$ and $\gamma \ge 0$, respectively.
For convenience, let $\mu^0 = \mu^{n+1} = 0$.
The subproblem, or \emph{pricing problem}, then asks for a path $p \in P$ with negative reduced cost.
By linear programming duality, such a path can be found by solving
\begin{equation*}
	\min_{p \in P} \sum_{(i,j)\in p} c_{ij} - \sum_{i \in p} \mu^i - g_p \gamma,
\end{equation*}
and comparing the optimal value to zero.

\paragraph{SPPRC}
The pricing problem is modeled as a Shortest Path Problem with Resource Constraints (SPPRC, \citep{IrnichDesaulniers2005-ShortestPathProblems}).
The SPPRC amounts to finding a shortest path from a source to a sink node, while satisfying resource constraints. The \emph{resources} are a vector of quantities $R\in\mathbb{R}^m$ that are updated along the route according to Resource Extension Functions (REFs) $F_{ij}:\mathbb{R}^m\rightarrow\mathbb{R}^m$ associated with each arc $(i,j)\in A$. Without loss of generality, the REFs represent \emph{minimal consumptions} of these quantities \citep{IrnichDesaulniers2005-ShortestPathProblems}.
A route is feasible if the resources remain within pre-specified bounds $l_i,u_i\in\mathbb{R}^m$ at each visited vertex $i\in V$.

We also assume that all REFs $F_{ij}$ are \emph{non-decreasing}. An REF $F$ is non-decreasing if and only if for any two resource vectors $R \le R'$ that are extended according to $F$, then $F(R) \leq F(R')$, where both comparisons are componentwise \citep{Irnich2008-ResourceExtensionFunctions}.
This is a common setting and the most interesting case for resource-robust applications, as we will show that the dominance rules in the labeling algorithm can be preserved in this case (Theorem~\ref{theorem:rr}).

\paragraph{Labeling Algorithm}
The SPPRC is solved with a \emph{labeling algorithm} that maintains a set of \emph{labels} that correspond to partial routes that start at vertex $0$.
Labels are \emph{extended} over all arcs to generate new labels, as defined by the REFs.
Infeasible labels are immediately removed, and \emph{dominance rules} are used to eliminate labels that are not Pareto-optimal.
When no more labels can be generated, we have found a route with minimum reduced cost.

A label is defined by the tuple $L = (\bar{c}, i, R)$.
For label $L$, we refer to these components as $\bar{c}(L)$, $i(L)$, and $R(L)$, respectively.
Here $\bar{c}(L)$ is the accumulated reduced cost of the partial route, $i(L)$ is the current endpoint, and $R(L)$ is the current resource vector.
We denote the extension of a label $L$ over $(i(L),j)\in A$ by $L\oplus j$.
Ignoring the valid inequality \eqref{eq:spp:vige} for the moment, $L\oplus j$ is defined as follows:
\begin{alignat}{2}
	\bar{c}(L\oplus j) &= \bar{c}(L) + c_{i(L),j} - \mu^{j}, \label{eq:cbar_ext}\\
	i(L\oplus j) &= j,\notag\\
	R(L\oplus j) &= \max\left\{l_j,F_{i(L),j}(R(L))\right\},\notag
\end{alignat}
where $F_{i(L),j}$ is the REF for arc $(i(L),j)$ and the maximum is taken element-wise over the resource vector.
If $R(L\oplus j)$ is not within the resource bounds $l_j$ and $u_j$, then the extension is infeasible, and $L\oplus j$ is discarded.

Dominance plays an important role in labeling algorithms, as it allows for removing labels that will not be part of the shortest path.
A label $L$ dominates label $L'$ if all feasible completions (extending the label until the sink node) of $L'$ are feasible for $L$ and $L$ has a lower or equal reduced cost than $L'$ for all of them \citep{PecinEtAl2017-ImprovedBranchCut}.
The non-decreasing REFs allow for the following dominance rule \citep{IrnichDesaulniers2005-ShortestPathProblems}: label $L$ dominates label $L'$ if all of the following are true:
\begin{alignat}{2}
	\bar{c}(L) &\le \bar{c}(L'), \label{eq:dom1}\\
	i(L) &= i(L'),\\
	R(L) &\le R(L'), \label{eq:dom3}
\end{alignat}
and at least one of the inequalities is strictly satisfied.

\section{Robust and Resource-Robust Applications}\label{sec:resource_robust}
In the branch-price-and-cut framework, care must be taken when including a valid inequality into the model.
The reduced cost update rule \eqref{eq:cbar_ext} must be updated to reflect the modified reduced cost that now includes the dual of the added cut.
Second, it needs to be verified that the dominance rules \eqref{eq:dom1}-\eqref{eq:dom3} still imply dominance, or adjustments may be necessary.
In this section, we define the notions of robust and resource-robust applications, which are two ways to easily and efficiently incorporate valid inequalities into the pricing problem.

\paragraph{Robust Applications}

If possible, a straightforward way to include a valid inequality is by projecting its dual onto the arcs of the underlying graph.
Definition \ref{def:r} states that a valid inequality admits a \emph{robust application} when such a projection is possible for a specific model.

\begin{mydef}[Robust Application]\label{def:r}
A valid inequality of the form $\sum_{p\in P}g_px_p\geq h$ admits a robust application for Formulation \eqref{eq:spp} if there exists constants $g_{ij}$, $(i,j)\in A$ such that $g_p = \sum_{(i,j)\in p}g_{ij}$ for all $p\in P$.
\end{mydef}

If Inequality~\eqref{eq:spp:vige} has a robust application, it can be incorporated by only updating reduced cost extension function \eqref{eq:cbar_ext} to
\begin{equation}
\label{eq:cbar_ext_R}
\bar{c}(L\oplus j)=\bar{c}(L)+c_{i(L),j}-\mu^j-g_{i(L),j}\gamma.
\end{equation}
It is straightforward to see that adding a constant to each arc does not affect dominance rules \eqref{eq:dom1}-\eqref{eq:dom3}, which continue to apply.
Implementing \eqref{eq:cbar_ext_R} is possible by pre-processing the arc costs used by the labeling algorithm and subtracting the new terms $g_{i(L),j}\gamma$.
After pre-processing, the time needed to extend or compare labels is independent of the number of cuts added.

When the specific model is clear from context, we call a valid inequality with a robust application a \emph{robust valid inequality} in line with the literature.
Similarly, if a valid inequality is not known to have a robust application in the given context, we may call the valid inequality \emph{non-robust}.
In cases where the context is not obvious, e.g., when comparing between papers, the new terminology is more precise, as it explicitly acknowledges that robustness depends on the specific model used.
It also formalizes what it means for valid inequalities to not ``change the structure of the pricing subproblem'' \cite{DeAragaoEtAl2003-IntegerProgramReformulation,FukasawaEtAl2006-RobustBranchCut,PessoaEtAl2008-RobustBranchCut}, i.e., only the cost parameters of the labeling algorithm need to be updated.

Definition~\ref{def:r} also helps to classify the extended capacity and triangle clique cuts \cite{PessoaEtAl2008-RobustBranchCut} mentioned in the introduction.
The authors show that the duals of these valid inequalities can be projected onto the arcs of the load-indexed layered graph of the original problem.
Hence, we say that these valid inequalities admit a robust application for the  \emph{layered graph formulation}.
For more information about layered graph approaches we refer to the survey by \cite{GouveiaEtAl2019-LayeredGraphApproaches}.

\paragraph{Resource-Robust Applications}
When inequalities do not admit a robust application, they are typically implemented through additional resources or specialized dominance rules, which may severely impact the performance of the pricing problem.
However, as observed by \citep{LiguoriEtAl2023-NonrobustStrongKnapsack} for the route load knapsack cuts, the impact on the pricing problem may be reduced if the resources that are necessary to calculate the cut contributions \emph{are already available in the labeling algorithm.}
To generalize this observation, we introduce the new concept of a \emph{resource-robust application}.
Cuts that admit a resource-robust application can efficiently be incorporated into the pricing problem in the context of a specific model \emph{and a specific set of resources}.

\begin{mydef}[Resource-robust Application]\label{def:rr}
A valid inequality of the form $\sum_{p\in P}g_px_p\geq h$ admits a resource-robust application for Formulation \eqref{eq:spp} if there exist non-increasing functions $g_{ij}:\mathbb{R}^m\rightarrow\mathbb{R}$, $(i,j)\in A$ such that
\[
g_p = \sum_{k=1}^{K}g_{i_ki_{k+1}}(R_k)
\]
for all $p=(i_{1},\ldots,i_{K+1}) \in P$ with resource vectors $(R_{1}, \ldots,$ $R_{K+1})$ along the route as defined by the SPPRC.
\end{mydef}

The merit of this definition is supported by the following result:

\begin{mytheorem}
	\label{theorem:rr}
    If a valid inequality has a resource-robust application, it can be incorporated into the pricing problem by changing the reduced cost extension function~\eqref{eq:cbar_ext} of extending label $L$ over arc $(i(L),j)$ to
	\begin{equation}
		\bar{c}(L\oplus j) = \bar{c}(L) + c_{i(L),j} - \mu^{j}  - g_{i(L),j}(R(L)) \gamma, \label{eq:cbar_ext_RR}
	\end{equation}
    with $\gamma$ the dual of the inequality.
	If \eqref{eq:dom1}-\eqref{eq:dom3} holds for two labels $L$ and $L'$, then $L$ dominates $L'$.
\end{mytheorem}
\begin{proof}
See the online supplement.
\end{proof}

That is, valid inequalities with a resource-robust application are easily implemented by only modifying the reduced cost extension function, while the standard dominance rules remain valid.
Also, no additional resources are introduced.
This is similar to robust applications, with the main difference that the cut contributions $g_{i(L),j}(R(L)) \gamma$ depend on the current resource vector $R(L)$.
The time to extend a label now grows linearly in the number of cuts, while the time to compare two labels remains constant.
While resource-robust applications incur some overhead compared to robust applications, this overhead is minor compared to the alternative of introducing additional resources, which may exponentially increase the number of non-dominated labels.

Similarly to robust applications, if the model and resources are clearly specified, we will call a valid inequality with a resource-robust application a \emph{resource-robust valid inequality}.
As previously mentioned, the route load knapsack cuts introduced by \cite{LiguoriEtAl2023-NonrobustStrongKnapsack} also fall into this category.
This follows directly from the original proof, in which the authors construct functions on the route load resource that satisfy Definition~\ref{def:rr}.

\section{New Resource-Robust Valid Inequalities}\label{sec:applications}
This paper applies the idea of resource-robust applications to the $ng$-route relaxation resources \citep{BaldacciEtAl2011-NewRouteRelaxation}.
The reason is twofold.
First, the $ng$-route relaxation is commonly used in state-of-the-art algorithms and the corresponding resources are thus readily available \citep{CostaEtAl2019-ExactBranchPrice,PessoaEtAl2020-GenericExactSolver}.
Second, the $ng$-route relaxation resources provide a wealth of information about the previously visited vertices that can be leveraged by the inequalities.

\citet{BaldacciEtAl2011-NewRouteRelaxation} introduced the $ng$-route relaxation to eliminate cycles within \emph{neighborhoods}.
For every vertex $i \in V'$, define a neighborhood $N_i \subseteq V'$ such that $i \in N_i$.

Next, $\lvert V' \rvert$ binary resources are introduced to indicate which vertices have been visited by partial route $L$ in the past.
It will be convenient to reason about these resources as a set $\Pi(L) \subseteq V'$ of vertices for which partial route $L$ has resource value one.

We will refer to $\Pi(L)$ as the \emph{$ng$-memory}.
The $ng$-memory is initialized as an empty set, and extensions from $L$ to $L'$ over $(i, j) \in A$ are updated according to
\begin{equation}
	\label{eq:ngmemoryupdate}
	\Pi(L') = (\Pi(L) \cap N_j)\cup\{j\}.
\end{equation}
Extensions to $j$ are only allowed if $j \notin \Pi(L)$.
It can be shown that the corresponding REF is non-decreasing and that the standard dominance rule~\eqref{eq:dom3} applies \citep{BaldacciEtAl2011-NewRouteRelaxation}.
This translates to $\Pi(L) \subseteq \Pi(L')$ in set notation.

It will be instructive to interpret the $ng$-memory $\Pi(L)$ as the vertices that the label $L$ \emph{remembers} visiting.
REF~\eqref{eq:ngmemoryupdate} then reads as follows: when label $L$ is extended to vertex $j \in V'$, the label forgets visiting the vertices outside of the neighborhood of $j$.
Next, the label commits the new vertex $j$ to memory.
The $ng$-route relaxation only allows extensions to vertices that the label does not remember ($j \notin \Pi(L)$), which eliminates cycles within neighborhoods.

It follows from the definition that increasing the size of the neighborhoods improves the $ng$-memory.
The extreme case $N_i = V'$ $\forall i \in V'$ in which all vertices are in the same neighborhood gives labels a \emph{perfect memory} of all previous visits.
This forces routes to be elementary, and we refer to the resources as \emph{elementarity resources} in this case \citep{ContardoEtAl2015-ReachingElementaryLower}.

\paragraph{Capacity Cuts}
Consider the case in which each customer $i \in V'$ has a demand $q_i \ge 0$ and each vehicle has a maximum capacity of $Q$.
\emph{Capacity cuts} in this setting are derived from the following observation: each subset $S \subseteq V'$ has to be served by at least $\left\lceil \frac{1}{Q}\sum_{i\in S}q_i\right\rceil$ routes to make sure there is enough capacity available.

The Strengthened Capacity Cuts (SCCs, \citep{BaldacciEtAl2004-ExactAlgorithmCapacitated}) are cuts that directly implement this observation:
\begin{equation}
	\label{eq:SCCs}
	\sum_{p \in P} \mathbb{I}(\lvert p \cap S\rvert \ge 1) x_p \geq \left\lceil \frac{1}{Q}\sum_{i\in S}q_i\right\rceil,
\end{equation}
with $\mathbb{I}(\cdot)$ the indicator function.
That is, every route that visits $S$ is counted once in the left-hand side summation.
The rounded Capacity Cuts (CCs, \citep{AugeratEtAl1998-SeparatingCapacityConstraints}) provide a weaker alternative:
\begin{equation}
	\label{eq:CCs}
	\sum_{p \in P} \sum_{(i, j) \in p} \mathbb{I}(i \notin S, j \in S) x_p \geq \left\lceil \frac{1}{Q}\sum_{i\in S}q_i\right\rceil.
\end{equation}
That is, routes are counted every time they enter the set $S$, even if they have been counted before.
This double counting results in a larger left-hand side and makes the cut weaker.
However, the dual costs can be projected onto the arcs using $g_{ij}=\mathbb{I}(i \notin S, j \in S)$, which allows a robust application for capacitated vehicle routing problems.
It is proven in the online supplement that the SCCs do not generally admit a robust application in this context, and as such we will refer to them as non-robust.

We can now define new cuts, called the \emph{$ng$-Capacity Cuts} ($ng$-CCs), that avoid double counting of route entries based on the information in the $ng$-memory.

\begin{mydef}[$ng$-Capacity Cuts]
	For a given path $p \in P$ with $K$ arcs, for $k \in \{1, \hdots, K+1\}$, let the $k$th label be denoted by $L_k$ and $i_k=i(L_k)$ such that $i_1=0$ and $i_{K+1}=n+1$.
	The $ng$-Capacity Cuts are defined as follows:
	\begin{equation}
		\label{eq:ngCCs}
		\sum_{p \in P} \sum_{k=1}^K \mathbb{I}(i_k \notin S, i_{k+1} \in S) \mathbb{I}(\Pi(L_k) \cap S = \emptyset) x_p \geq \left\lceil \frac{1}{Q}\sum_{i\in S}q_i\right\rceil.
	\end{equation}
That is, routes that enter the set $S$ are only counted if the label does not remember visiting $S$ in the past, i.e., $\Pi(L_k) \cap S = \emptyset$.
\end{mydef}

Figure \ref{fig:ngCCexample} shows an example for a route $p=(0,1,2,3,4,5,\allowbreak n+1)$ and a subset $S=\{1,3,5\}$.
Assume that the neighborhoods are given by $N_1=N_3=N_4=N_5=\{1,3,4,5\}$ and $N_2=\{2,3\}$.
As $p$ travels into $S$ at least once, the coefficient in SCC~\eqref{eq:SCCs} is 1.
Route $p$ travels into $S$ three times, so the coefficient in CC~\eqref{eq:CCs} is 3.
According to $ng$-memory, the route enters $S$ \emph{two times}: first when traveling from 0 to 1 (as the $ng$-memory for this extension is empty), and second when traveling from 2 to 3 (the visit to $1$ is not remembered because $1\notin N_2$).
Traveling from 4 to 5, the $ng$-memory equals $\{3,4\}$, so this entry is not counted.
It follows that the $ng$-CC \eqref{eq:ngCCs} has coefficient 2.

\begin{figure}[t]
	\centering
	\includegraphics[width=6cm]{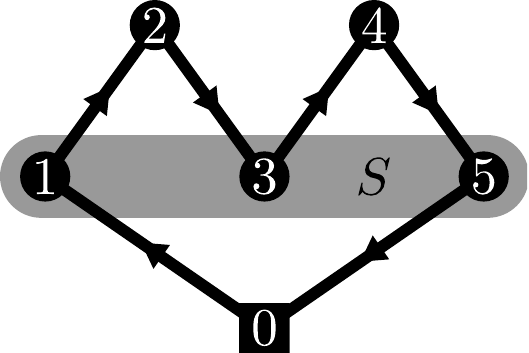}
	\caption{Example instance for the $ng$-CCs.}
	\label{fig:ngCCexample}
\end{figure}

In general, the coefficients of an SCC, $ng$-CC, and CC for the same route satisfy the following inequalities:
\begin{align}
\hspace{-0.5cm}\mathbb{I}(|p\cap S|\geq 1)&\leq\sum_{k=1}^K\mathbb{I}(i_k\notin S,i_{k+1}\in S)\mathbb{I}(\Pi(L_k)\cap S=\emptyset)\notag\\&\leq \sum_{k=1}^K\mathbb{I}(i_k\notin S,i_{k+1}\in S)\text{.}\label{eq:ngCC:correct}
\end{align}

The first inequality follows from the fact that if $p$ visits $S$, then the first entry $(i_k, i_{k+1})$ into $S$ satisfies both $i_k\notin S$, $i_{k+1}\in S$ (by definition) and $\Pi(L_k)\cap S=\emptyset$ (because this is the first entry into $S$).
This shows that the $ng$-CCs are not stronger than the SCCs, and are therefore valid cuts.
The second inequality in \eqref{eq:ngCC:correct} is trivial, and shows that the $ng$-CCs dominate the CCs.
To prove that the $ng$-CCs admit a resource-robust application when the $ng$-route relaxation is used, it is sufficient to show that
\begin{equation}
	\label{eq:ngCCg}
	g_{ij}(R) = \mathbb{I}(i \notin S, j \in S)\mathbb{I}(\Pi(R) \cap S = \emptyset)
\end{equation}
is non-increasing, where $\Pi(R)\subseteq V'$ is the $ng$-memory associated with resource vector $R\in\mathbb{R}^m$.
This follows immediately from the fact that $R\leq R'$ implies $\Pi(R)\subseteq\Pi(R')$, which implies $\mathbb{I}(\Pi(R) \cap S = \emptyset) \geq \mathbb{I}(\Pi(R') \cap S = \emptyset)$.

\paragraph{$k$-Path Cuts}
The same derivation can be used to obtain resource-robust $k$-path cuts \citep{KohlEtAl1999-2PathCuts}.
Let  $S \subseteq V'$ be a subset of customers, and let $k(S) \ge 1$ be a lower bound on the number of routes that are needed to serve $S$ in any feasible solution.
The $k$-path cuts are given by
\begin{equation*}
	\sum_{p \in P} \sum_{(i, j) \in p} \mathbb{I}(i \notin S, j \in S) x_p \geq k(S).
\end{equation*}

Note that the left-hand side is identical to that of the CCs~\eqref{eq:CCs}.
Hence, we may define $ng$-robust $k$-path cuts with $g_{ij}(R)$ defined as in \eqref{eq:ngCCg}.
As special cases we obtain $ng$-robust subtour elimination constraints \citep{DantzigEtAl1954-SolutionLargeScale} and $ng$-robust 2-path cuts \citep{KohlEtAl1999-2PathCuts}.
The same idea can also be applied to the generalized $k$-path cuts presented in \citep{DesaulniersEtAl2008-TabuSearchPartial}.

\paragraph{Strong Degree Constraints}

For a given customer $v \in V'$, the Strong Degree Constraint (SDC, \citep{ContardoEtAl2014-ExactAlgorithmBased}) is given by
\begin{equation*}
	\sum_{p\in P} \mathbb{I}(v \in p) x_p \geq 1,
\end{equation*}
where the coefficient $\mathbb{I}(v \in p)$ indicates that customer $v$ is visited by path $p$.
\citet{ContardoEtAl2014-ExactAlgorithmBased} introduce these cuts as an alternative to enforcing elementarity in the pricing problem.
That is, not including the elementarity resources and using these cuts instead gives the same main problem lower bound and better algorithmic performance.

We construct a variant that has a resource-robust application when the $ng$-route relaxation is used by observing that $\mathbb{I}(v \in p)$ can be interpreted as only counting a visit to $v \in V'$ if $v$ has not been visited before.
This condition is relaxed by only counting a visit to $v$ if the label does not \emph{remember} customer $v$ according to its $ng$-memory.
This results in an $ng$-SDC that is defined by the arc contributions
\begin{equation*}
	g_{ij}(R) = \mathbb{I}(j = v) \mathbb{I}(v \notin \Pi(R)).
\end{equation*}

Interestingly, while SDCs can be used as an alternative to enforcing elementarity, we show in the online supplement that $ng$-SDCs can be used as an alternative to enforcing $ng$-routes.
Given the computational success of the SDCs, using $ng$-SDCs in this way is an interesting topic for future research.

\section{Computational Experiments}
\label{sec:compexp}

To demonstrate the potential of resource-robust valid inequalities, we solve the Capacitated Vehicle Routing Problem (CVRP) with 1) only the robust CCs, 2) with the robust CCs and the resource-robust $ng$-CCs, or 3) with robust CCs and the non-robust SCCs.
The CVRP corresponds to the basic model~\eqref{eq:spp:obj}-\eqref{eq:spp:constr}, \eqref{eq:spp:binary} with the additional constraint that each route $p\in P$ satisfies a capacity constraint \citep{TothVigo2014-VehicleRoutingProblems}.

\subsection{Algorithm}
A straightforward Branch-Price-and-Cut (BPC) algorithm was implemented in C++ according to the setting described in Section~\ref{sec:setting}.
This section summarizes the main components, and for details we refer to the survey by \citet{CostaEtAl2019-ExactBranchPrice}.

\paragraph{Labeling}
A standard labeling algorithm is implemented with load $q(L)$ as an additional resource that has resource bounds $[0,Q]$ and REF $f_{ij}^{\text{load}}(q)=q+q_j$.
Following \citep{Pecin2014-ExactAlgorithmsCapacitated}, dominance checks are only performed between labels with the same current load.
The $ng$-route relaxation is implemented with neighborhoods of size 10 by selecting the nearest customers.
The dual contributions of CCs are projected onto the arcs, while dual contributions of $ng$-CCs are calculated according to~\eqref{eq:ngCCg}. For the SCCs, a new binary resource is introduced for every cut with a non-zero dual in the pricing problem. This resource keeps track of whether the corresponding set $S \subseteq V'$ has been visited or not. With these additional resources in place, standard label dominance rules \eqref{eq:dom1}-\eqref{eq:dom3} still apply.

\paragraph{Heuristic Pricing}
Heuristic pricing is applied through aggressive dominance and arc elimination \citep{CostaEtAl2019-ExactBranchPrice,DesaulniersEtAl2008-TabuSearchPartial}.
This allows for discovering promising routes quickly.
Eventually, the exact pricing problem is solved to guarantee optimality.

\paragraph{Cut Separation}
We leverage the CVRPSEP package by \citet{Lysgaard2003-CvrpsepPackageSeparation} to heuristically find violated valid inequalities for fractional solutions $x_p \ge 0$ $\forall p \in P$ encountered during the BPC process.
CVRPSEP was written to separate CCs, and may therefore fail to find violations of the stronger $ng$-CCs.
To remedy this situation, we modify CVRPSEP and the arc flows to find more violated $ng$-CCs.
We give the details of our separation algorithm in the online supplement.
CVRPSEP is a heuristic, so for a fair comparison the same procedure is used to separate CCs, $ng$-CCs and SCCs.

When only CCs are used, all separated cuts are added as CCs.
When $ng$-CCs are used, for each violated inequality, it is determined whether the cut is added as a CC or as an $ng$-CC.
If both have similar strength, the cut is added as a CC to avoid the overhead of evaluating~\eqref{eq:ngCCg}.
In particular, for CC and $ng$-CC violations $v^{CC}$ and $v^{ng}$, respectively, the cut is added as an $ng$-CC if $v^{ng} \ge \max\{0.2,\: v^{CC} + 0.1\}$, as a CC if $v^{CC} \ge 0.1$, and ignored otherwise.
The same procedure is used for SCCs instead of $ng$-CCs.

\paragraph{Strong Branching}
We implement strong branching on arc flows, similar to \citet{Ropke2012-BranchingDecisionsBranch} and \citet{PecinEtAl2017-ImprovedBranchCut}.
For a given arc $(i,j)\in A$, the score of a potential branch is given by $\Delta= 0.75\min\{\Delta^+,\Delta^-\}+0.25\max\{\Delta^+,\Delta^-\}$, where $\Delta^+$ and $\Delta^-$ are the objective values after forcing the flow on arc $(i,j)$ to one or zero, respectively. The arc to branch on is chosen as follows:
\begin{enumerate}
\setlength{\itemsep}{0pt}
\setlength{\parskip}{0pt}
	\item Select the 30 most fractional arcs.
	\item For these arcs, calculate the $\Delta$-score without generating any additional routes. Select the five arcs with the highest score.
	\item For these arcs, calculate the $\Delta$-score while generating additional routes with heuristic pricing only. Select the arc with the highest score to branch on.
\end{enumerate}

\paragraph{Other Details}
All experiments are run on an Intel Xeon  W-2123 $3.6$GHz processor and 16GB of RAM.
The algorithm is run on a single thread, and a time limit of one hour per instance is imposed.
All linear programs are solved with CPLEX $22.1.0$.

\subsection{Test Instances}
The benefit of using $ng$-CCs over using only robust or non-robust cuts is demonstrated on the \texttt{A}, \texttt{B} and \texttt{P} CVRP benchmark instances of \citet{Augerat1995-ApprochePolyedraleDu}.
Preliminary experiments suggest that these particular cuts work well when customer demand is relatively high compared to the vehicle capacity.
To fully explore this observation, we perform experiments on instances for which each customer demand is scaled by a factor $\alpha \in \{1, 1.25, 1.5, 1.75, 2\}$ (or set to the vehicle capacity if exceeded).
This results in 15 different instance classes that we distinguish by adding the scale factor to the name of the original class, e.g., \texttt{P175} refers to benchmark \texttt{P} with its demand scaled by factor $\alpha=1.75$.
The instance files can be found at {\small{\url{https://github.com/YNHoogendoorn/DataResults/tree/main/HoogendoornDalmeijer2024}}}.

\subsection{Results}

\begin{figure}[t]
	\centering
    \includegraphics[width=0.83\linewidth]{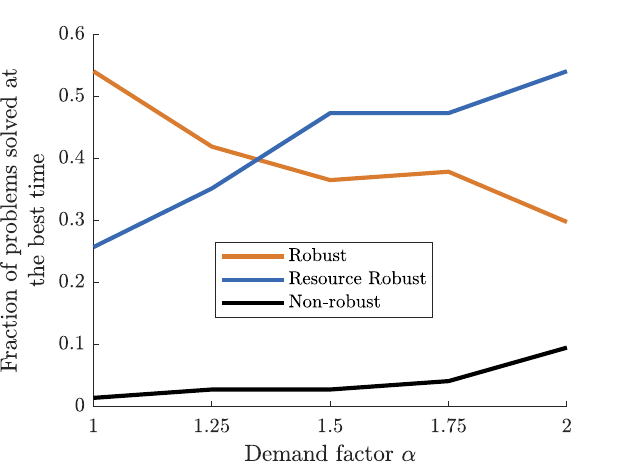}
	\caption{Fraction of all instances that were solved within the fastest time, for each demand factor $\alpha$.}
	\label{fig:fastest}
\end{figure}

\begin{table*}[p]
\resizebox{.91\textwidth}{!}{%
\begin{tabular}{ld{4.2}d{3.0}d{2.2}d{1.2}d{4.2}d{3.0}d{3.0}d{2.2}d{1.2}d{4.2}d{3.0}d{3.0}d{2.2}d{1.2}}
\toprule
&\multicolumn{4}{c}{Robust}&\multicolumn{5}{c}{Resource Robust}&\multicolumn{5}{c}{Non Robust}\\
\cmidrule(lr){2-5}\cmidrule(lr){6-10}\cmidrule(lr){11-15}
\multicolumn{1}{c}{Instance}&\multicolumn{1}{c}{Lower Bound}&\multicolumn{1}{c}{CCs}&\multicolumn{1}{c}{Seconds}&\multicolumn{1}{c}{Gap \%}&\multicolumn{1}{c}{Lower Bound}&\multicolumn{1}{c}{CCs}&\multicolumn{1}{c}{$ng$-CCs}&\multicolumn{1}{c}{Seconds}&\multicolumn{1}{c}{Gap \%}&\multicolumn{1}{c}{Lower Bound}&\multicolumn{1}{c}{CCs}&\multicolumn{1}{c}{SCCs}&\multicolumn{1}{c}{Seconds}&\multicolumn{1}{c}{Gap \%}\\
\midrule
\texttt{B2-n31-k5}&1100&61&0.18&0.00&1099.00&62&12&0.18&0.09&1100&63&14&0.29&0.00\\
\texttt{B2-n34-k5}&1262&47&0.20&0.00&1262&47&88&0.76&0.00&1262&47&99&0.49&0.00\\
\texttt{B2-n35-k5}&1608.50&62&0.69&1.21&1613.38&48&31&0.71&0.91&1613.84&69&34&7.25&0.88\\
\texttt{B2-n38-k6}&1335.58&115&0.43&0.18&1336.00&127&42&0.82&0.15&1336.00&146&42&0.78&0.15\\
\texttt{B2-n39-k5}&952&62&0.70&0.00&952&48&32&1.08&0.00&952&41&41&4.73&0.00\\
\texttt{B2-n41-k6}&1475.31&85&0.48&0.11&1476.57&88&34&0.60&0.03&1476.85&72&31&4.85&0.01\\
\texttt{B2-n43-k6}&1218.36&140&1.39&0.38&1219.83&155&53&1.72&0.26&1221.37&124&96&15.08&0.13\\
\texttt{B2-n44-k7}&1587.77&116&0.67&0.27&1588.71&109&85&1.19&0.21&1589.67&117&98&17.29&0.15\\
\texttt{B2-n45-k5}&1143.30&144&1.30&0.15&1144.11&99&56&1.93&0.08&1144.17&99&61&21.54&0.07\\
\texttt{B2-n45-k6}&1156.33&142&0.82&0.75&1156.68&112&68&1.23&0.72&1156.33&106&44&1.58&0.75\\
\texttt{B2-n50-k7}&1204&112&0.91&0.00&1204&79&41&1.72&0.00&1204&86&43&25.24&0.00\\
\texttt{B2-n50-k8}&2170.11&253&1.36&0.50&2173.06&261&75&1.81&0.37&2174.43&172&121&3.99&0.30\\
\texttt{B2-n51-k7}&1749.16&94&0.65&2.11&1753.02&84&61&0.84&1.88&1753.02&82&65&2.08&1.88\\
\texttt{B2-n52-k7}&1241.06&158&1.19&0.08&1241.62&180&119&2.85&0.03&1241.62&136&110&6.12&0.03\\
\texttt{B2-n56-k7}&1200.00&269&2.15&0.17&1200.50&207&93&2.68&0.12&1200.45&144&52&27.22&0.13\\
\texttt{B2-n57-k7}&2059.66&208&2.48&0.26&2060.30&165&45&4.03&0.23&2061.09&165&81&19.41&0.19\\
\texttt{B2-n57-k9}&2885.42&131&1.18&0.30&2887.62&96&13&1.30&0.22&2888.06&107&20&3.41&0.21\\
\texttt{B2-n63-k10}&2709.81&77&1.68&1.04&2709.77&70&7&1.84&1.04&2717.33&97&35&6.09&0.76\\
\texttt{B2-n64-k9}&1528.21&264&3.42&0.71&1529.21&262&60&3.76&0.64&1531.82&185&68&35.34&0.47\\
\texttt{B2-n66-k9}&2293.96&126&3.50&0.96&2295.93&119&52&3.92&0.87&2295.15&131&61&12.97&0.91\\
\texttt{B2-n67-k10}&1761.54&209&2.74&0.82&1764.68&169&103&3.03&0.64&1768.35&175&121&36.92&0.43\\
\texttt{B2-n68-k9}&2285.46&182&5.67&\dsh&2288.91&215&60&6.28&\dsh&2290.80&152&56&37.82&\dsh\\
\texttt{B2-n78-k10}&2099.79&188&3.98&0.68&2103.07&182&93&8.47&0.52&2102.92&207&109&260.05&0.53\\
\bottomrule
\end{tabular}}
\caption{Root Node Results for the \texttt{B2} instances.}
\label{tbl:results:root}
\end{table*}

\begin{table*}[p]
\resizebox{.91\textwidth}{!}{%
\begin{tabular}{ld{4.0}d{4.2}d{4.0}rrd{4.0}d{4.2}d{4.0}d{3.0}rrd{4.0}d{4.2}d{4.0}d{3.0}rr}
\toprule
&\multicolumn{5}{c}{Robust}&\multicolumn{6}{c}{Resource Robust}&\multicolumn{6}{c}{Non Robust}\\
\cmidrule(lr){2-6}\cmidrule(lr){7-12}\cmidrule(lr){13-18}
\multicolumn{1}{c}{Instance}&\multicolumn{1}{c}{Upper B.}&\multicolumn{1}{c}{Lower B.}&\multicolumn{1}{c}{CCs}&\multicolumn{1}{c}{Seconds}&\multicolumn{1}{c}{Nodes}&\multicolumn{1}{c}{Upper B.}&\multicolumn{1}{c}{Lower B.}&\multicolumn{1}{c}{CCs}&\multicolumn{1}{c}{$ng$-CCs}&\multicolumn{1}{c}{Seconds}&\multicolumn{1}{c}{Nodes}&\multicolumn{1}{c}{Upper B.}&\multicolumn{1}{c}{Lower B.}&\multicolumn{1}{c}{CCs}&\multicolumn{1}{c}{SCCs}&\multicolumn{1}{c}{Seconds}&\multicolumn{1}{c}{Nodes}\\
\midrule
\texttt{B2-n31-k5}&1100&=&61&0&1&1100&=&68&13&1&3&1100&=&63&14&0&1\\
\texttt{B2-n34-k5}&1262&=&47&0&1&1262&=&47&88&1&1&1262&=&47&99&0&1\\
\texttt{B2-n35-k5}&1628&=&370&800&767&1628&=&121&71&154&209&1631&1624.68&104&64&\texttt{timeout}&63\\
\texttt{B2-n38-k6}&1338&=&120&9&25&1338&=&127&42&7&13&1338&=&146&44&18&23\\
\texttt{B2-n39-k5}&952&=&62&1&1&952&=&48&32&1&1&952&=&41&41&5&1\\
\texttt{B2-n41-k6}&1477&=&146&19&40&1477&=&107&38&7&13&1477&=&73&33&39&15\\
\texttt{B2-n43-k6}&1223&=&172&37&31&1223&=&186&63&86&43&1223&=&152&100&213&13\\
\texttt{B2-n44-k7}&1592&=&195&48&71&1592&=&139&95&32&31&1592&=&190&125&670&23\\
\texttt{B2-n45-k5}&1145&=&180&27&31&1145&=&108&73&14&13&1147&1144.89&102&62&\texttt{timeout}&8\\
\texttt{B2-n45-k6}&1165&=&619&902&301&1165&=&354&147&469&203&1165&=&246&108&458&161\\
\texttt{B2-n50-k7}&1204&=&112&1&1&1204&=&79&41&2&1&1204&=&86&43&25&1\\
\texttt{B2-n50-k8}&2181&=&866&1437&466&2181&=&567&194&862&345&2181&=&332&201&915&165\\
\texttt{B2-n51-k7}&1786&=&380&1098&1091&1786&=&312&135&818&869&1789&1769.75&178&109&\texttt{timeout}&174\\
\texttt{B2-n52-k7}&1242&=&169&39&43&1242&=&184&120&37&23&1242&=&150&112&104&17\\
\texttt{B2-n56-k7}&1202&=&348&172&61&1202&=&253&101&110&29&1202&=&178&59&630&17\\
\texttt{B2-n57-k7}&2065&=&348&678&223&2065&=&263&99&401&84&\dsh&2064.22&235&115&\texttt{timeout}&55\\
\texttt{B2-n57-k9}&2894&=&478&351&219&2894&=&422&96&245&173&2894&=&341&82&195&65\\
\texttt{B2-n63-k10}&2738&2735.50&1465&\texttt{timeout}&1195&2738&=&920&289&1881&895&2743&2730.64&418&193&\texttt{timeout}&243\\
\texttt{B2-n64-k9}&1539&=&844&1652&505&1539&=&493&132&679&173&1542&1536.19&299&94&\texttt{timeout}&58\\
\texttt{B2-n66-k9}&\dsh&2311.84&801&\texttt{timeout}&701&2316&=&1000&244&3073&492&\dsh&2311.64&353&142&\texttt{timeout}&118\\
\texttt{B2-n67-k10}&1776&=&875&3443&929&1776&=&350&187&479&147&1776&=&351&205&2229&129\\
\texttt{B2-n68-k9}&\dsh&2295.66&492&\texttt{timeout}&722&\dsh&2297.88&412&133&\texttt{timeout}&517&\dsh&2293.42&195&71&\texttt{timeout}&24\\
\texttt{B2-n78-k10}&\dsh&2107.45&549&\texttt{timeout}&632&2114&2112.39&482&182&\texttt{timeout}&668&\dsh&2103.71&267&123&\texttt{timeout}&8\\
\bottomrule
\end{tabular}}
\caption{Branch-Price-and-Cut Results for the \texttt{B2} instances.}
\label{tbl:results:bpc}
\end{table*}

\begin{figure}[t]
	\centering
    \includegraphics[width=0.83\linewidth]{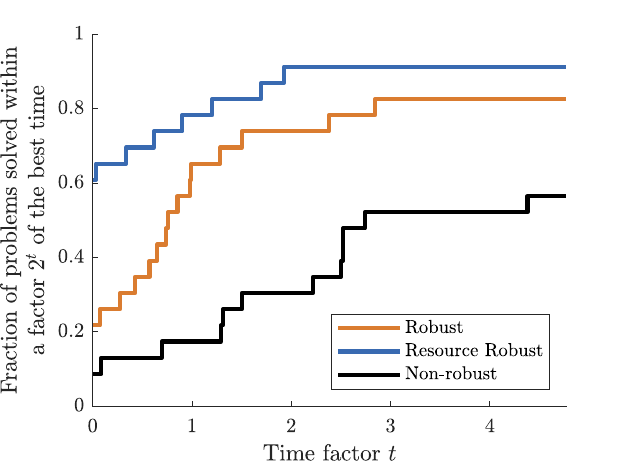}
	\caption{Performance Profiles BPC Solution Times for the \texttt{B2} instances.}
	\label{fig:performance_profile}
\end{figure}

Detailed results for all 15 instance classes are presented in the online supplement, and Figure~\ref{fig:fastest} provides a summary.
The figure shows the percentage of instances for which each configuration (only robust, resource-robust, or non-robust) provides the fastest solution time for different demand factors $\alpha$.
For regular demand $\alpha=1$, using an algorithm with resource-robust $ng$-CCs already results in the fastest solution time for more than 25\% of the instances.
The relative performance further improves as demand increases, and for $\alpha \ge 1.5$ using resource-robust cuts leads to the fastest solution time in most cases.

To examine the strong performance of the algorithm with resource-robust cuts, we zoom in on the \texttt{B2} instances for the remainder of this section.
This class represents a prototypical high-demand case in which the $ng$-CCs are very effective, and the conclusions extend to the other classes that are detailed in the online supplement.
The online supplement also shows that it is rare for $ng$-CCs to significantly \emph{decrease} performance, making it relatively safe to enable these cuts by default.

Table~\ref{tbl:results:root} compares the root-node performance for the \texttt{B2} instance class using only the robust CCs, including the resource-robust $ng$-CCs, and including the non-robust SCCs.
For each instance, the table presents the root node lower bound, the number of cuts added of each type, the total time in seconds to solve the root node, and the root node gap calculated against the best integer solution found by any of the algorithms.
As expected, in almost all cases the robust algorithm has the weakest lower bound and shortest solving time, while the non-robust algorithm has the strongest lower bound and longest solving time.
It can be seen that the resource-robust algorithm provides a very favorable trade-off with short solving times that are closer to that of the robust algorithm, and strong lower bounds that are closer to that of the non-robust algorithm.

In the full BPC framework, Table~\ref{tbl:results:bpc} shows that the resource-robust algorithm leads to significantly better performance compared to only using robust or non-robust cuts.

In addition to the previous statistics, the BPC table includes the number of nodes in the search tree, and matching bounds are indicated by an equals sign ($=$).

It can be seen that the stronger cuts result in substantially fewer nodes in the search tree compared to the robust algorithm, and substantially less overhead compared to the non-robust algorithm, resulting in a decrease of the overall solution time.
This is especially notable for the more difficult instances: the solution times of \texttt{B2-n64-k9}, and \texttt{B2-n67-k10} were each more than halved, and \texttt{B2-n63-k10} and \texttt{B2-n66-k9} were only solved by the resource-robust algorithm.

Figure~\ref{fig:performance_profile} summarizes the performance with \citet{DolanMore2002-BenchmarkingOptimizationSoftware} performance profiles (higher is better).
The figure shows that the resource-robust algorithm provides the best solution time for 60\% of the instances ($t=0$).
Even if the other algorithms are sped up by a factor two ($t=1$) they fall short of this performance.
The horizontal lines indicate that the robust and non-robust algorithm were able to solve 19/23 and 14/23 instances within the one hour time limit, respectively, while the resource-robust algorithm reached 21/23.

\section{Conclusions}
\label{sec:conclusion}

There are many interesting directions to explore in future research.
While this paper provides a proof of concept for the effectiveness of resource-robust applications, it would be valuable to test the new cuts within state-of-the-art implementations, combining the new $ng$-capacity cuts and $ng$-strong degree constraints in the same framework.
It is especially interesting to see how these cuts would interact with enhancements such as bidirectional labeling, unreachable customers, and route enumeration \citep{CostaEtAl2019-ExactBranchPrice}, and further research may be necessary to truly integrate with these other features.

There are a number of opportunities to build on the $ng$-robust valid inequalities introduced in this paper.
A heuristic was provided for separating violated $ng$-capacity cuts, but dedicated separation algorithms may be able to find better cuts.
Also, the $ng$-robust cuts may be extended to other concepts of memory, such as the arc-based $ng$-memory introduced by \cite{BulhoesEtAl2018-BranchPriceAlgorithm}.
Another idea is to use a dynamic $ng$-route relaxation \citep{RobertiMingozzi2014-DynamicNgPath}, which grows the neighborhoods over time, to strengthen the resource-robust cuts dynamically.

The setting in this paper is also applicable to problems on acyclic graphs, including generalized assignment, bin packing, and vector packing \citep{PessoaEtAl2020-GenericExactSolver}.
As there are no cycles to prevent, there was previously no reason to introduce $ng$-resources for these problems.
But with resource-robust inequalities, \emph{$ng$-resources can be introduced for the purpose of strengthening the cuts}.
This provides an interesting direction for future research.

Finally, there is the potential for new resource-robust variants or completely new resource-robust valid inequalities that are not necessarily based on $ng$-resources.
The route load knapsack cuts by \cite{LiguoriEtAl2023-NonrobustStrongKnapsack} represent a first step in this direction.

\bibliographystyle{cas-model2-names}
\bibliography{references}

\end{document}


\maketitle

\pagenumbering{roman}

\section{Proof of Theorem 1}
In this section, we give a proof of Theorem 1. For convenience, the theorem statement is repeated below.
\begin{theorem}
	\label{theorem:rr}
    If a valid inequality has a resource-robust application, it can be incorporated into the pricing problem by changing the reduced cost extension function~(2) of extending label $L$ over arc $(i(L),j)$ to
	\begin{equation}
		\bar{c}(L\oplus j) = \bar{c}(L) + c_{i(L),j} - \mu^{j}  - g_{i(L),j}(R(L)) \gamma, \tag{6}\label{eq:cbar_ext_RR}
	\end{equation}
    with $\gamma$ the dual of the inequality.
	If (3)-(5) holds for two labels $L$ and $L'$, then $L$ dominates $L'$.
\end{theorem}
\begin{proof}
Consider a feasible route $p=(0=i_1,i_2,\ldots,i_K,i_{K+1}=n+1)$ with labels $L_{1},\ldots,L_{K+1}$ along the route. Applying \eqref{eq:cbar_ext_RR} over each of the label extensions, we get $\bar{c}(L_{K+1})=\sum_{k=1}^{K}\big[c_{i_ki_{k+1}}-\mu^{i_k}-g_{i_ki_{k+1}}(R(L_k))\gamma\big] =\sum_{(i,j)\in p}c_{ij}-\sum_{i\in p}\mu^{i}- g_p\gamma$ by Definition~2, which is the reduced cost of route $p$. Note that we can ignore $\mu^{i_{K+1}}$ as it equals 0 per definition. Thus, extending labels using \eqref{eq:cbar_ext_RR} gives the correct reduced costs of a route.

Now assume we have two labels $L$ and $L'$ such that (3)-(5) holds. As $R(L)\leq R(L')$ and the REFs are non-decreasing, each feasible completion for $L'$ is feasible for $L$. To see this, take a single feasible extension to $j$. We get that $j=i(L\oplus j)=i(L'\oplus j)$ and $R(L\oplus j)=\max\{l_j,F_{(i(L),j)}(R(L))
\}\leq \max\{l_j,F_{(i(L'),j)}(R(L'))\}=R(L'\oplus j)$.
Furthermore, $g_{ij}$ is non-increasing by Definition~2, so we have that $\bar{c}(L\oplus j)=\bar{c}(L)+c_{i(L),j}-\mu^j-g_{i(L),j}(R(L))\gamma\leq\bar{c}(L')+c_{i(L'),j}-\mu^j-g_{i(L'),j}(R(L'))\gamma=\bar{c}(L'\oplus j)$. Thus, (3)-(5) also holds for $L\oplus j$ and $L'\oplus j$. Repeating this argument until the labels are completed, we can conclude that the reduced cost of any completion of $L$ is lower than that of $L'$, and thus that $L$ dominates $L'$.
\end{proof}

\section{Non-robustness of SCCs}
For a subset $S\subseteq V'$, the SCCs are defined as
\[
\sum_{p\in P}\mathbb{I}(|p\cap S|\geq 1)x_p\geq\left\lceil\frac{1}{Q}\sum_{i\in S}q_i\right\rceil\text{.}
\]
We show by means of a counterexample that the SCCs do not generally admit a robust application for the capacitated vehicle routing problem.
Consider an instance with $n=4$ customers, with demands $q_1=q_2=q_3=2$ and $q_4=1$, and vehicle capacity $Q=5$.
Next, we define six routes and their SCC coefficients for $S=\{1,2,3\}$:
\begin{align*}
    \text{Route 1: } & (0,4,5) & \mathbb{I}(|p\cap S|\geq 1) &= 0\\
    \text{Route 2: } & (0,1,4,2,5) & \mathbb{I}(|p\cap S|\geq 1) &= 1\\
    \text{Route 3: } & (0,1,4,5) & \mathbb{I}(|p\cap S|\geq 1) &= 1\\
    \text{Route 4: } & (0,4,2,5) & \mathbb{I}(|p\cap S|\geq 1) &= 1\\
    \text{Route 5: } & (0,1,3,5) & \mathbb{I}(|p\cap S|\geq 1) &= 1\\
    \text{Route 6: } & (0,2,3,5) & \mathbb{I}(|p\cap S|\geq 1) &= 1
\end{align*}%
These coefficients define the following SCC for $S=\{1,2,3\}$: $x_2 + x_3 + x_4 + x_5 + x_6 \ge 2$.
Next, we compare the two solutions $x^1=(x^1_1,x^1_2,x^1_3,x^1_4,x^1_5,x^1_6)=\left(\frac{1}{2},\frac{1}{2},0,0,\frac{1}{2},\frac{1}{2}\right)$ and $x^2=(x^2_1,x^2_2,x^2_3,x^2_4,x^2_5,\allowbreak x^2_6)=\left(0,0,\frac{1}{2},\frac{1}{2},\frac{1}{2},\frac{1}{2}\right)$.
It is straightforward to verify that both $x^1$ and $x^2$ are feasible solutions to the linear relaxation of (1), and that $x^1$ violates the SCC above while $x^2$ adheres to it.
Note that $x^1$ and $x^2$ have the same arc flows, that is, $\sum_{p:(i,j)\in p}x^1_p=\sum_{p:(i,j)\in p}x^2_p$ for all $(i,j)\in A$.
The fact that the arc flows are the same, but only $x^1$ violates the SCC, shows that the dual contribution of the cut cannot be captured by only projecting onto the arcs.
It follows that the SCCs do not generally admit a robust application in this context.

\section{Using $ng$-SDCs to enforce $ng$-routes}
For a given vertex $v\in V'$, the $ng$-SDCs are defined as
\[
\sum_{p\in P} \sigma_p^v x_p \geq 1\text{,}
\]
where $\sigma_p^v$ is the number of times route $p$ visits vertex $v$ while not remembering $v$ in its $ng$-memory. In this section, we show that enforcing the $ng$-SDCs for all $v\in V'$ is equivalent to enforcing $ng$-routes. 

Let us consider some $v\in V'$ and some feasible (possibly relaxed) solution $(x_p)_{p\in P}$. In this case, we allow $P$ to contain routes that are not $ng$-routes. Subtracting the $ng$-SDC corresponding to $v$ from (1b) gives
\[
\sum_{p\in P}(a^v_p - \sigma^v_p)x_p \leq 0\text{.}
\]
As $a^v_p - \sigma^v_p\geq 0$, this inequality implies $x_p=0$ for all $p\in P$ with $a^v_p - \sigma^v_p > 0$. These are precisely the routes that visit $v$ at least once when $v$ is already in the $ng$-memory of that route. Thus, enforcing the $ng$-SDCs for all $v\in V'$ implies that, if $x_p>0$, any vertex $v$ is visited only when $v$ is not in $ng$-memory. In other words, only $ng$-routes are allowed to take positive values.

\section{Separation algorithm}
In this section, we detail our separation algorithm to separate $ng$-CCs.

For a fractional solution $x_p \ge 0$ $\forall p \in P$, we propose to scale down the flow of each route $p$ by a factor $\kappa_p$, defined as $\min\left\{\kappa,\left\lceil\frac{l(p)}{2}\right\rceil\right\}$ for some $\kappa \ge 1$, with $l(p)$ the number of locations visited by $p$.
As this results in a violation of (1b), we restore flow balance by introducing two dummy nodes $i'$ and $i''$ per node $i\in V'$ (see Figure~\ref{fig:restore_flow_balance}). Then, we run CVRPSEP on this scaled network.
CVRPSEP is also modified to not terminate early if a violated cut it found, as not all cuts in the scaled network necessarily correspond to violated $ng$-CCs in the original network.
In Proposition \ref{prop:scaleNetwork} we show that scaling by a factor $\kappa_p$ allows all $ng$-CCs to be found that enter the set $S \subseteq V'$ at most $\kappa$ times.

We implement a strategy where CVRPSEP is first called for $\kappa=1$, and is increased by one when no more violated cuts can be found.
The maximum value of $\kappa$ is set to 4, as preliminary experiments have shown that no additional $ng$-CCs are found when $\kappa\geq 5$.

\begin{figure}[t]
	\centering
	\includegraphics[width=0.95\linewidth]{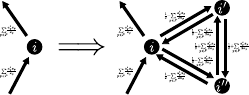}
	\caption{Restoring flow balance for scaled flows.}
	\label{fig:restore_flow_balance}
\end{figure}

\begin{proposition}\label{prop:scaleNetwork}
Given a solution to the continuous relaxation of the CVRP $(x_p)_{p\in P}$ and some $\kappa\in\mathbb{N}_{>0}$, if a violated $ng$-CC in the original network corresponds to a subset $S\subseteq V'$ that is entered at most $\kappa$ times by every non-zero route $p$ in this solution, then the CC corresponding to $S$ in the flow-reduced network $(\bar{x}_p)_{p\in P}=\left(\frac{x_p}{\kappa_p}\right)_{p\in P}$ is violated as well.
\end{proposition}
\begin{proof}
Let $(x_p)_{p\in P}$ be a solution to the continuous relaxation of the CVRP\text{.} The flow-reduced solution is defined as $(\bar{x}_p)_{p\in P}=\left(\frac{x_p}{\kappa_p}\right)_{p\in P}$, with $\kappa_p=\min\left\{\kappa,\left\lceil\frac{l(p)}{2}\right\rceil\right\}$ and $l(p)$ the number of non-depot locations route $p$ visits. The flow is reduced by a factor of $\kappa_p$ as each route can only visit $S$ at most $\kappa$ times per assumption and $\left\lceil\frac{l(p)}{2}\right\rceil$ times per definition. We will show that, if the $ng$-CC with subset $S$ is violated by the solution $(x_p)_{p\in P}$ such that all $p\in P$ with $x_p>0$ enter $S$ at most $\kappa$ times, then the CC with subset $S$ is violated by $(\bar{x}_p)_{p\in P}$.

Let $\alpha_p$, $\beta_p$ and $\gamma_p$ be the coefficient of $x_p$ in the CC, $ng$-CC and SCC of subset $S$ respectively. That is, $\alpha_p=\sum_{(i,j)\in p}\mathbb{I}(i\notin S,j\in S)$, $\beta_p=\sum_{k=1}^K\mathbb{I}(i_k\notin S,i_{k+1}\in S)\mathbb{I}(\Pi(L_k)\cap S=\emptyset)$ and $\gamma_p=\mathbb{I}(|p\cap S|\geq 1)$. We have that $\alpha_p\leq\kappa_p$ for all $p$ with $x_p>0$. As $\beta_p\geq\gamma_p=\min\{\alpha_p,1\}$, we have $\kappa_p\beta_p\geq\min\{\kappa_p\alpha_p,\kappa_p\}\geq\alpha_p$. Since the $ng$-CC of $S$ is violated in the original network, the following sequence of inequalities hold:
\[
\sum_{p\in P}\alpha_p\bar{x}_p=\sum_{p:x_p>0}\frac{\alpha_p}{\kappa_p}x_p\leq \sum_{p:x_p>0}\beta_p x_p < \left\lceil\frac{1}{Q}\sum_{i\in S}q_i\right\rceil\text{.}
\]
Thus, the CC of $S$ is violated in the reduced network.
\end{proof}

\newpage
\section{Additional results}
\subsection{Performance profiles}
The performance profiles listed below are the Dolan and Mor\'{e} [13] performance profiles for the 15 instance classes, equivalent to Figure 3 of the main text.
\begin{figure}[h]
	\centering
 \begin{minipage}{.5\textwidth}
	\centering
    \includegraphics[width=0.75\linewidth]{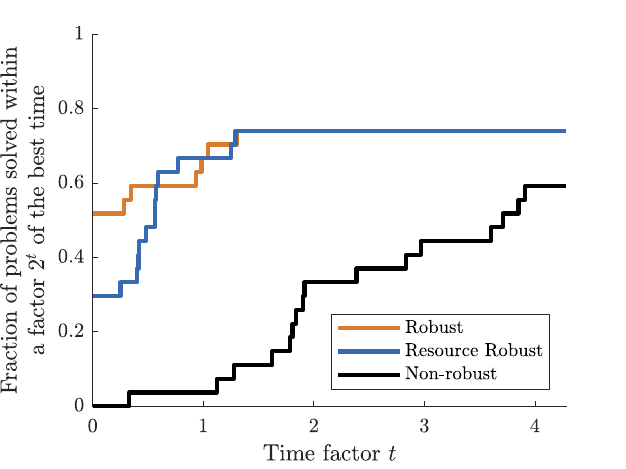}
	\caption{Performance Profiles \texttt{A} instances.}
	\label{fig:performance_profile_A}
 \end{minipage}%
  \begin{minipage}{.5\textwidth}
	\centering
    \includegraphics[width=0.75\linewidth]{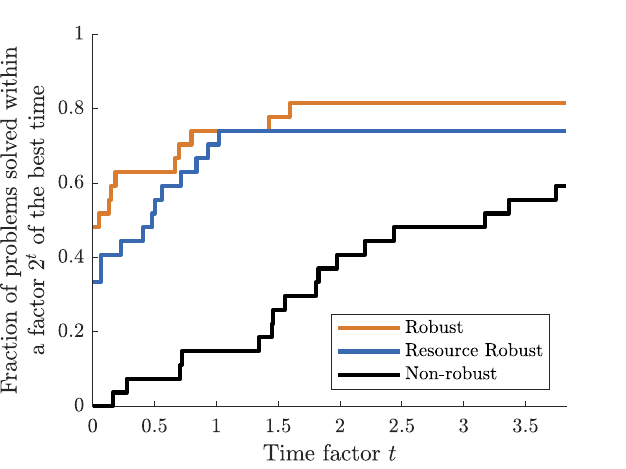}
	\caption{Performance Profiles \texttt{A125} instances.}
	\label{fig:performance_profile_A125}
 \end{minipage}
\end{figure}
\begin{figure}[h]
	\centering
 \begin{minipage}{.5\textwidth}
	\centering
    \includegraphics[width=0.75\linewidth]{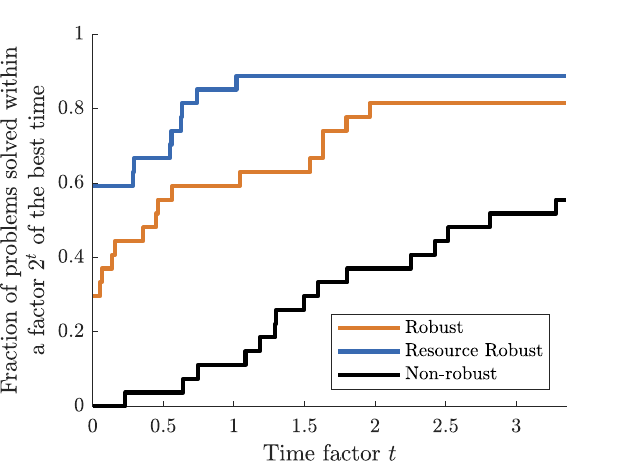}
	\caption{Performance Profiles \texttt{A15} instances.}
	\label{fig:performance_profile_A15}
 \end{minipage}%
  \begin{minipage}{.5\textwidth}
	\centering
    \includegraphics[width=0.75\linewidth]{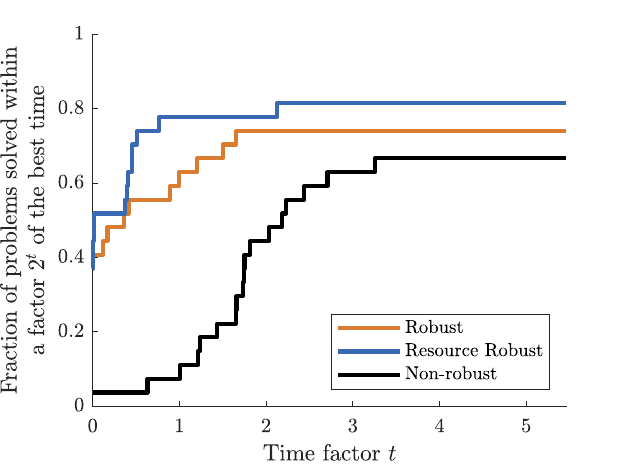}
	\caption{Performance Profiles \texttt{A175} instances.}
	\label{fig:performance_profile_A175}
 \end{minipage}
\end{figure}
\begin{figure}[h]
	\centering
    \includegraphics[width=0.375\linewidth]{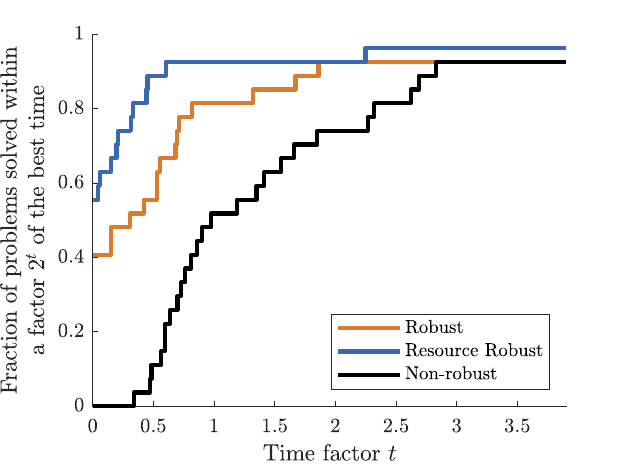}
	\caption{Performance Profiles \texttt{A2} instances.}
	\label{fig:performance_profile_A2}
\end{figure}
\newpage

\begin{figure}[h]
	\centering
 \begin{minipage}{.5\textwidth}
	\centering
    \includegraphics[width=0.8\linewidth]{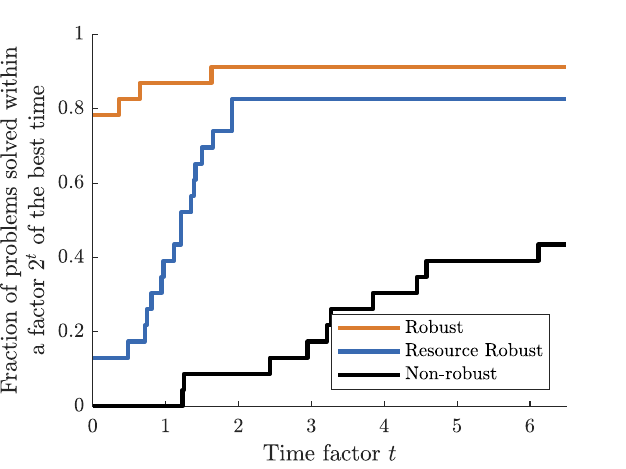}
	\caption{Performance Profiles \texttt{B} instances.}
	\label{fig:performance_profile_B}
 \end{minipage}%
  \begin{minipage}{.5\textwidth}
	\centering
    \includegraphics[width=0.8\linewidth]{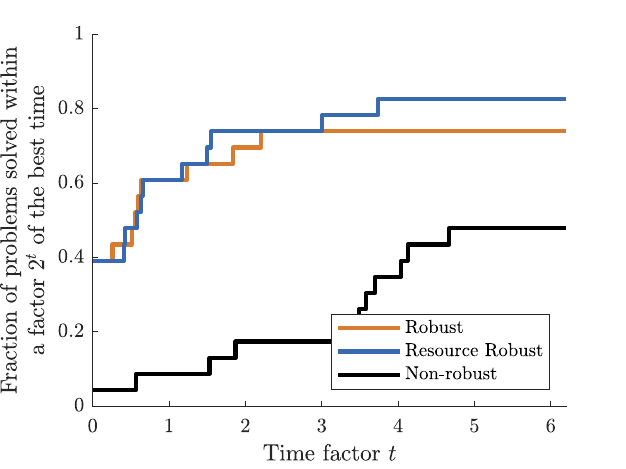}
	\caption{Performance Profiles \texttt{B125} instances.}
	\label{fig:performance_profile_B125}
 \end{minipage}
\end{figure}
\begin{figure}[h]
	\centering
 \begin{minipage}{.5\textwidth}
	\centering
    \includegraphics[width=0.8\linewidth]{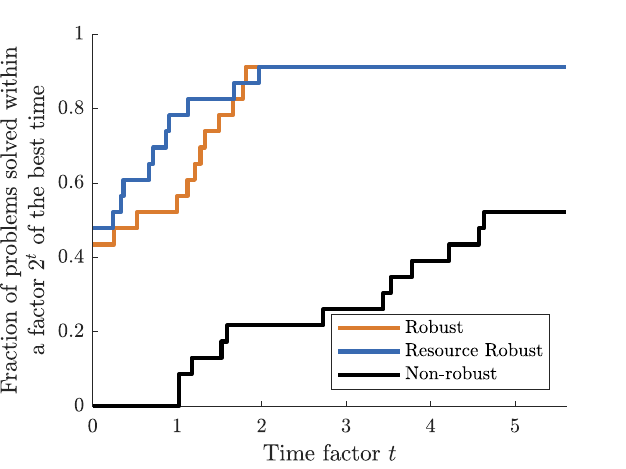}
	\caption{Performance Profiles \texttt{B15} instances.}
	\label{fig:performance_profile_B15}
 \end{minipage}%
  \begin{minipage}{.5\textwidth}
	\centering
    \includegraphics[width=0.8\linewidth]{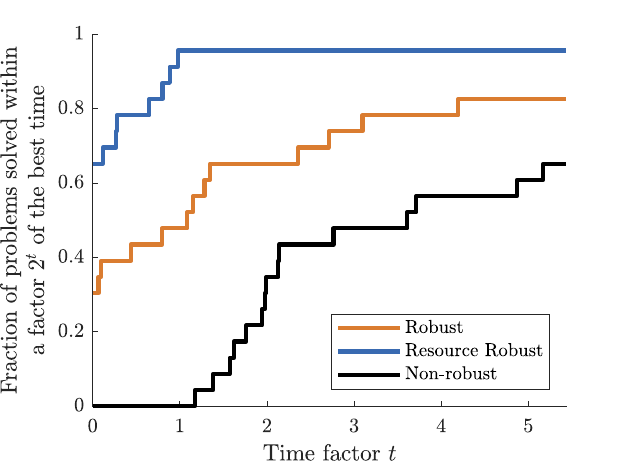}
	\caption{Performance Profiles \texttt{B175} instances.}
	\label{fig:performance_profile_B175}
 \end{minipage}
\end{figure}
\begin{figure}[h]
	\centering
    \includegraphics[width=0.4\linewidth]{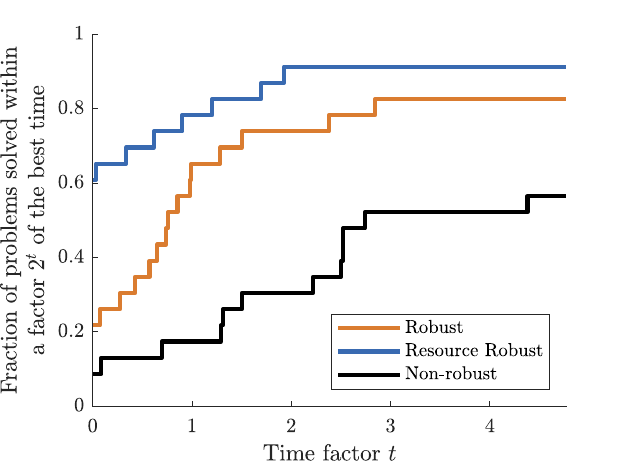}
	\caption{Performance Profiles \texttt{B2} instances.}
	\label{fig:performance_profile_B2}
\end{figure}
\newpage

\begin{figure}[h]
	\centering
 \begin{minipage}{.5\textwidth}
	\centering
    \includegraphics[width=0.8\linewidth]{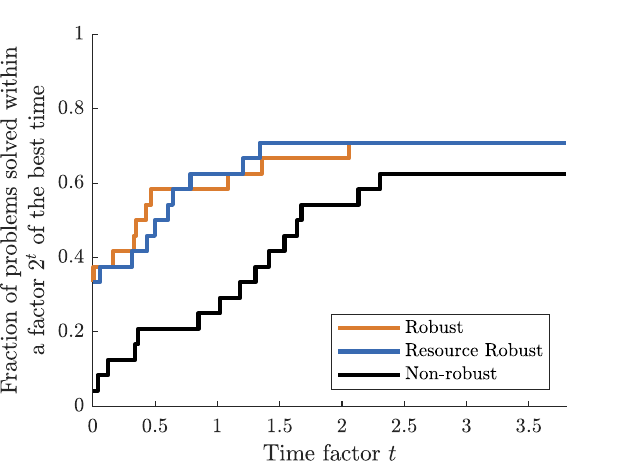}
	\caption{Performance Profiles \texttt{P} instances.}
	\label{fig:performance_profile_P}
 \end{minipage}%
  \begin{minipage}{.5\textwidth}
	\centering
    \includegraphics[width=0.8\linewidth]{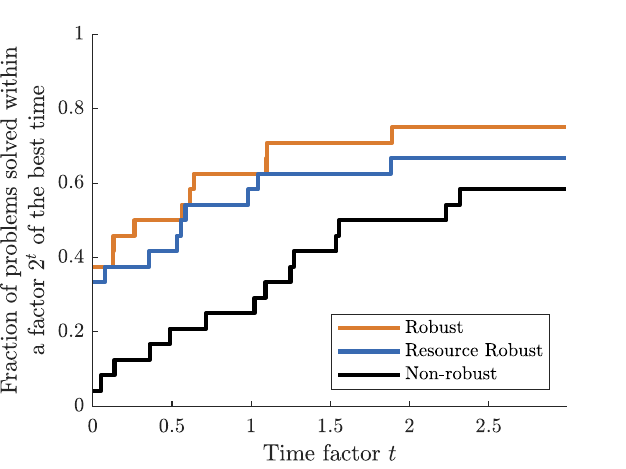}
	\caption{Performance Profiles \texttt{P125} instances.}
	\label{fig:performance_profile_P125}
 \end{minipage}
\end{figure}
\begin{figure}[h]
	\centering
 \begin{minipage}{.5\textwidth}
	\centering
    \includegraphics[width=0.8\linewidth]{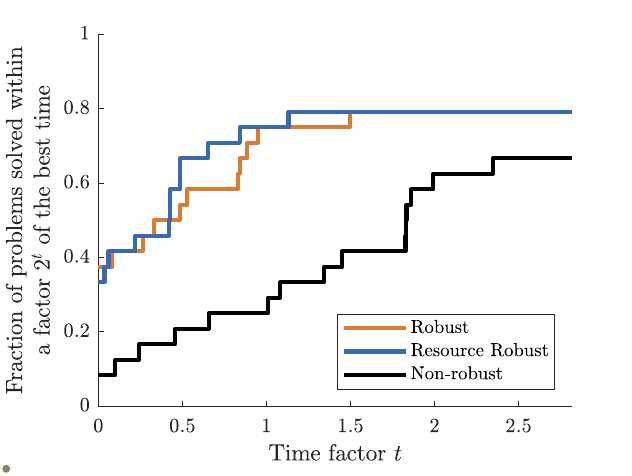}
	\caption{Performance Profiles \texttt{P15} instances.}
	\label{fig:performance_profile_P15}
 \end{minipage}%
  \begin{minipage}{.5\textwidth}
	\centering
    \includegraphics[width=0.8\linewidth]{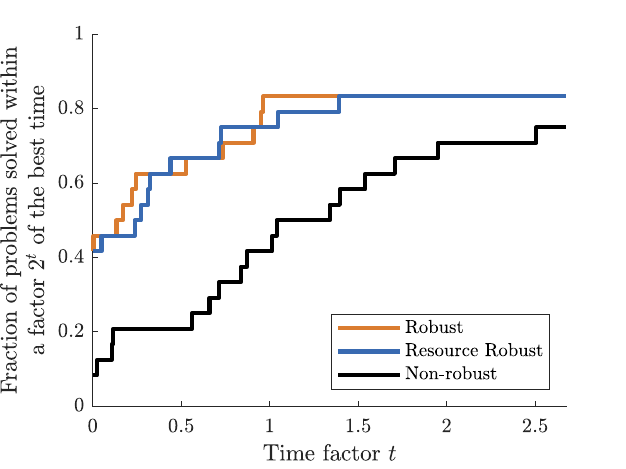}
	\caption{Performance Profiles \texttt{P175} instances.}
	\label{fig:performance_profile_P175}
 \end{minipage}
\end{figure}
\begin{figure}[h]
	\centering
    \includegraphics[width=0.4\linewidth]{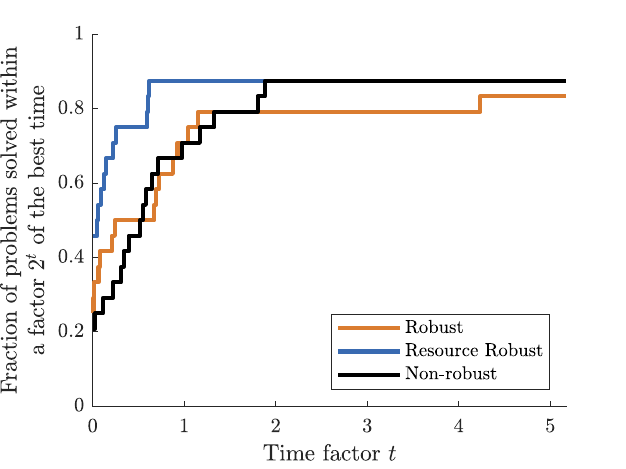}
	\caption{Performance Profiles \texttt{P2} instances.}
	\label{fig:performance_profile_P2}
\end{figure}
\newpage

\subsection{Root node and branch-price-and-cut results}
The tabels listed below show the root node results and the branch-price-and-cut results for all 15 instance classes, equivalent to Tables 1 and 2 of the main text.

\begin{table*}[h]
\resizebox{\textwidth}{!}{%
\begin{tabular}{ld{4.2}d{3.0}d{2.2}d{1.2}d{4.2}d{3.0}d{3.0}d{2.2}d{1.2}d{4.2}d{3.0}d{3.0}d{2.2}d{1.2}}
\toprule
&\multicolumn{4}{c}{Robust}&\multicolumn{5}{c}{Resource Robust}&\multicolumn{5}{c}{Non Robust}\\
\cmidrule(lr){2-5}\cmidrule(lr){6-10}\cmidrule(lr){11-15}
\multicolumn{1}{c}{Instance}&\multicolumn{1}{c}{Lower Bound}&\multicolumn{1}{c}{CCs}&\multicolumn{1}{c}{Seconds}&\multicolumn{1}{c}{Gap \%}&\multicolumn{1}{c}{Lower Bound}&\multicolumn{1}{c}{CCs}&\multicolumn{1}{c}{$ng$-CCs}&\multicolumn{1}{c}{Seconds}&\multicolumn{1}{c}{Gap \%}&\multicolumn{1}{c}{Lower Bound}&\multicolumn{1}{c}{CCs}&\multicolumn{1}{c}{SCCs}&\multicolumn{1}{c}{Seconds}&\multicolumn{1}{c}{Gap \%}\\
\midrule
\texttt{A-n32-k5}&784&68&2.02&0.00&784&60&22&3.45&0.00&784&66&22&4.41&0.00\\
\texttt{A-n33-k5}&661&57&1.13&0.00&661&36&218&1.70&0.00&661&36&234&3.92&0.00\\
\texttt{A-n33-k6}&740.07&94&1.37&0.26&740.07&69&120&1.75&0.26&740.07&69&122&2.50&0.26\\
\texttt{A-n34-k5}&775.00&105&2.12&0.39&775.00&50&49&4.65&0.39&775.00&53&45&10.54&0.39\\
\texttt{A-n36-k5}&798.55&110&4.96&0.06&799&73&215&11.22&0.00&799&72&245&27.17&0.00\\
\texttt{A-n37-k5}&666.63&132&6.21&0.36&666.94&132&267&7.77&0.31&667.00&136&275&31.48&0.30\\
\bottomrule
\end{tabular}}
\caption{Root Node Results, \texttt{A} instances.}
\label{tbl:results:root:A}
\end{table*}

\begin{table*}[h]
\resizebox{\textwidth}{!}{%
\begin{tabular}{ld{4.0}d{4.2}d{4.0}rrd{4.0}d{4.2}d{4.0}d{3.0}rrd{4.0}d{4.2}d{4.0}d{3.0}rr}
\toprule
&\multicolumn{5}{c}{Robust}&\multicolumn{6}{c}{Resource Robust}&\multicolumn{6}{c}{Non Robust}\\
\cmidrule(lr){2-6}\cmidrule(lr){7-12}\cmidrule(lr){13-18}
\multicolumn{1}{c}{Instance}&\multicolumn{1}{c}{Upper B.}&\multicolumn{1}{c}{Lower B.}&\multicolumn{1}{c}{CCs}&\multicolumn{1}{c}{Seconds}&\multicolumn{1}{c}{Nodes}&\multicolumn{1}{c}{Upper B.}&\multicolumn{1}{c}{Lower B.}&\multicolumn{1}{c}{CCs}&\multicolumn{1}{c}{$ng$-CCs}&\multicolumn{1}{c}{Seconds}&\multicolumn{1}{c}{Nodes}&\multicolumn{1}{c}{Upper B.}&\multicolumn{1}{c}{Lower B.}&\multicolumn{1}{c}{CCs}&\multicolumn{1}{c}{SCCs}&\multicolumn{1}{c}{Seconds}&\multicolumn{1}{c}{Nodes}\\
\midrule
\texttt{A-n32-k5}&784&=&68&2&1&784&=&60&22&3&1&784&=&66&22&4&1\\
\texttt{A-n33-k5}&661&=&57&1&1&661&=&36&218&2&1&661&=&36&234&4&1\\
\texttt{A-n33-k6}&742&=&94&19&19&742&=&84&120&25&15&742&=&89&128&23&13\\
\texttt{A-n34-k5}&778&=&114&35&13&778&=&57&52&48&13&778&=&57&47&124&13\\
\texttt{A-n36-k5}&799&=&110&28&5&799&=&73&215&11&1&799&=&72&245&27&1\\
\texttt{A-n37-k5}&669&=&134&71&11&669&=&138&267&105&13&669&=&146&276&929&17\\
\bottomrule
\end{tabular}}
\caption{Branch-Price-and-Cut Results, \texttt{A} instances.}
\label{tbl:results:bpc:A}
\end{table*}
\newpage

\newpage
\begin{table*}[h]
\resizebox{\textwidth}{!}{%
\begin{tabular}{ld{4.2}d{3.0}d{2.2}d{1.2}d{4.2}d{3.0}d{3.0}d{2.2}d{1.2}d{4.2}d{3.0}d{3.0}d{2.2}d{1.2}}
\toprule
&\multicolumn{4}{c}{Robust}&\multicolumn{5}{c}{Resource Robust}&\multicolumn{5}{c}{Non Robust}\\
\cmidrule(lr){2-5}\cmidrule(lr){6-10}\cmidrule(lr){11-15}
\multicolumn{1}{c}{Instance}&\multicolumn{1}{c}{Lower Bound}&\multicolumn{1}{c}{CCs}&\multicolumn{1}{c}{Seconds}&\multicolumn{1}{c}{Gap \%}&\multicolumn{1}{c}{Lower Bound}&\multicolumn{1}{c}{CCs}&\multicolumn{1}{c}{$ng$-CCs}&\multicolumn{1}{c}{Seconds}&\multicolumn{1}{c}{Gap \%}&\multicolumn{1}{c}{Lower Bound}&\multicolumn{1}{c}{CCs}&\multicolumn{1}{c}{SCCs}&\multicolumn{1}{c}{Seconds}&\multicolumn{1}{c}{Gap \%}\\
\midrule
\texttt{A125-n32-k5}&892&37&0.79&0.00&892&34&31&1.59&0.00&892&34&31&2.17&0.00\\
\texttt{A125-n33-k5}&754.68&79&0.87&0.31&754.82&65&31&1.53&0.29&755.05&45&51&4.13&0.26\\
\texttt{A125-n33-k6}&858.14&53&0.60&1.15&860.63&49&27&1.23&0.86&860.63&49&27&12.12&0.86\\
\texttt{A125-n34-k5}&885.93&67&1.23&1.14&886.54&57&28&2.10&1.07&886.54&60&24&5.52&1.07\\
\texttt{A125-n36-k5}&905.29&55&1.28&0.08&905.29&45&31&1.71&0.08&906&48&32&36.48&0.00\\
\texttt{A125-n37-k5}&748.66&132&2.30&0.31&749.62&119&16&4.32&0.18&749.62&119&16&33.06&0.18\\
\bottomrule
\end{tabular}}
\caption{Root Node Results, \texttt{A125} instances.}
\label{tbl:results:root:A125}
\end{table*}

\begin{table*}[h]
\resizebox{\textwidth}{!}{%
\begin{tabular}{ld{4.0}d{4.2}d{4.0}rrd{4.0}d{4.2}d{4.0}d{3.0}rrd{4.0}d{4.2}d{4.0}d{3.0}rr}
\toprule
&\multicolumn{5}{c}{Robust}&\multicolumn{6}{c}{Resource Robust}&\multicolumn{6}{c}{Non Robust}\\
\cmidrule(lr){2-6}\cmidrule(lr){7-12}\cmidrule(lr){13-18}
\multicolumn{1}{c}{Instance}&\multicolumn{1}{c}{Upper B.}&\multicolumn{1}{c}{Lower B.}&\multicolumn{1}{c}{CCs}&\multicolumn{1}{c}{Seconds}&\multicolumn{1}{c}{Nodes}&\multicolumn{1}{c}{Upper B.}&\multicolumn{1}{c}{Lower B.}&\multicolumn{1}{c}{CCs}&\multicolumn{1}{c}{$ng$-CCs}&\multicolumn{1}{c}{Seconds}&\multicolumn{1}{c}{Nodes}&\multicolumn{1}{c}{Upper B.}&\multicolumn{1}{c}{Lower B.}&\multicolumn{1}{c}{CCs}&\multicolumn{1}{c}{SCCs}&\multicolumn{1}{c}{Seconds}&\multicolumn{1}{c}{Nodes}\\
\midrule
\texttt{A125-n32-k5}&892&=&37&1&1&892&=&34&31&2&1&892&=&34&31&2&1\\
\texttt{A125-n33-k5}&757&=&81&6&5&757&=&82&38&7&5&757&=&50&51&50&13\\
\texttt{A125-n33-k6}&868&=&190&238&211&868&=&101&50&79&95&868&=&99&57&1057&61\\
\texttt{A125-n34-k5}&896&=&98&94&91&896&=&74&43&98&71&896&=&69&31&334&51\\
\texttt{A125-n36-k5}&906&=&102&37&13&906&=&47&34&33&9&906&=&48&32&36&1\\
\texttt{A125-n37-k5}&751&=&155&76&19&751&=&127&20&44&9&751&=&127&20&237&9\\
\bottomrule
\end{tabular}}
\caption{Branch-Price-and-Cut Results, \texttt{A125} instances.}
\label{tbl:results:bpc:A125}
\end{table*}
\newpage

\begin{table*}[h]
\resizebox{\textwidth}{!}{%
\begin{tabular}{ld{4.2}d{3.0}d{2.2}d{1.2}d{4.2}d{3.0}d{3.0}d{2.2}d{1.2}d{4.2}d{3.0}d{3.0}d{2.2}d{1.2}}
\toprule
&\multicolumn{4}{c}{Robust}&\multicolumn{5}{c}{Resource Robust}&\multicolumn{5}{c}{Non Robust}\\
\cmidrule(lr){2-5}\cmidrule(lr){6-10}\cmidrule(lr){11-15}
\multicolumn{1}{c}{Instance}&\multicolumn{1}{c}{Lower Bound}&\multicolumn{1}{c}{CCs}&\multicolumn{1}{c}{Seconds}&\multicolumn{1}{c}{Gap \%}&\multicolumn{1}{c}{Lower Bound}&\multicolumn{1}{c}{CCs}&\multicolumn{1}{c}{$ng$-CCs}&\multicolumn{1}{c}{Seconds}&\multicolumn{1}{c}{Gap \%}&\multicolumn{1}{c}{Lower Bound}&\multicolumn{1}{c}{CCs}&\multicolumn{1}{c}{SCCs}&\multicolumn{1}{c}{Seconds}&\multicolumn{1}{c}{Gap \%}\\
\midrule
\texttt{A15-n32-k5}&1032.61&95&1.09&0.23&1034.39&95&12&1.49&0.06&1035&101&11&7.53&0.00\\
\texttt{A15-n33-k5}&838.50&40&0.49&0.06&838.50&34&22&0.50&0.06&839&41&41&2.95&0.00\\
\texttt{A15-n33-k6}&946.40&81&0.44&0.49&946.40&72&30&0.49&0.49&946.40&69&31&0.83&0.49\\
\texttt{A15-n34-k5}&995&49&0.49&0.00&995&38&22&0.60&0.00&995&36&28&1.40&0.00\\
\texttt{A15-n36-k5}&1026.43&41&1.09&0.83&1027.02&44&9&1.72&0.78&1027.39&43&16&4.54&0.74\\
\texttt{A15-n37-k5}&810&119&1.01&0.00&810&119&146&1.49&0.00&810&119&146&3.07&0.00\\
\bottomrule
\end{tabular}}
\caption{Root Node Results, \texttt{A15} instances.}
\label{tbl:results:root:A15}
\end{table*}

\begin{table*}[h]
\resizebox{\textwidth}{!}{%
\begin{tabular}{ld{4.0}d{4.2}d{4.0}rrd{4.0}d{4.2}d{4.0}d{3.0}rrd{4.0}d{4.2}d{4.0}d{3.0}rr}
\toprule
&\multicolumn{5}{c}{Robust}&\multicolumn{6}{c}{Resource Robust}&\multicolumn{6}{c}{Non Robust}\\
\cmidrule(lr){2-6}\cmidrule(lr){7-12}\cmidrule(lr){13-18}
\multicolumn{1}{c}{Instance}&\multicolumn{1}{c}{Upper B.}&\multicolumn{1}{c}{Lower B.}&\multicolumn{1}{c}{CCs}&\multicolumn{1}{c}{Seconds}&\multicolumn{1}{c}{Nodes}&\multicolumn{1}{c}{Upper B.}&\multicolumn{1}{c}{Lower B.}&\multicolumn{1}{c}{CCs}&\multicolumn{1}{c}{$ng$-CCs}&\multicolumn{1}{c}{Seconds}&\multicolumn{1}{c}{Nodes}&\multicolumn{1}{c}{Upper B.}&\multicolumn{1}{c}{Lower B.}&\multicolumn{1}{c}{CCs}&\multicolumn{1}{c}{SCCs}&\multicolumn{1}{c}{Seconds}&\multicolumn{1}{c}{Nodes}\\
\midrule
\texttt{A15-n32-k5}&1035&=&126&19&21&1035&=&108&12&6&7&1035&=&101&11&8&1\\
\texttt{A15-n33-k5}&839&=&49&2&5&839&=&43&28&2&3&839&=&41&41&3&1\\
\texttt{A15-n33-k6}&951&=&111&14&29&951&=&86&34&5&7&951&=&82&37&11&7\\
\texttt{A15-n34-k5}&995&=&49&0&1&995&=&38&22&1&1&995&=&36&28&1&1\\
\texttt{A15-n36-k5}&1035&=&83&71&63&1035&=&78&31&69&41&1035&=&71&35&394&39\\
\texttt{A15-n37-k5}&810&=&119&1&1&810&=&119&146&1&1&810&=&119&146&3&1\\
\bottomrule
\end{tabular}}
\caption{Branch-Price-and-Cut Results, \texttt{A15} instances.}
\label{tbl:results:bpc:A15}
\end{table*}
\newpage

\begin{table*}[h]
\resizebox{\textwidth}{!}{%
\begin{tabular}{ld{4.2}d{3.0}d{2.2}d{1.2}d{4.2}d{3.0}d{3.0}d{2.2}d{1.2}d{4.2}d{3.0}d{3.0}d{2.2}d{1.2}}
\toprule
&\multicolumn{4}{c}{Robust}&\multicolumn{5}{c}{Resource Robust}&\multicolumn{5}{c}{Non Robust}\\
\cmidrule(lr){2-5}\cmidrule(lr){6-10}\cmidrule(lr){11-15}
\multicolumn{1}{c}{Instance}&\multicolumn{1}{c}{Lower Bound}&\multicolumn{1}{c}{CCs}&\multicolumn{1}{c}{Seconds}&\multicolumn{1}{c}{Gap \%}&\multicolumn{1}{c}{Lower Bound}&\multicolumn{1}{c}{CCs}&\multicolumn{1}{c}{$ng$-CCs}&\multicolumn{1}{c}{Seconds}&\multicolumn{1}{c}{Gap \%}&\multicolumn{1}{c}{Lower Bound}&\multicolumn{1}{c}{CCs}&\multicolumn{1}{c}{SCCs}&\multicolumn{1}{c}{Seconds}&\multicolumn{1}{c}{Gap \%}\\
\midrule
\texttt{A175-n32-k5}&1173&117&0.61&0.00&1173&113&57&0.84&0.00&1173&102&65&2.88&0.00\\
\texttt{A175-n33-k5}&944.63&69&0.64&0.57&944.61&60&8&0.56&0.57&945.31&32&12&1.55&0.50\\
\texttt{A175-n33-k6}&1031.20&37&0.29&0.17&1032.67&27&11&0.45&0.03&1033&25&9&0.29&0.00\\
\texttt{A175-n34-k5}&1087&40&0.31&0.00&1087&41&26&0.52&0.00&1087&40&31&1.40&0.00\\
\texttt{A175-n36-k5}&1140.74&35&0.81&0.99&1141.42&33&9&0.96&0.93&1141.45&33&11&3.13&0.92\\
\texttt{A175-n37-k5}&887.72&106&1.09&0.82&889.00&101&27&2.01&0.67&889.40&99&31&9.10&0.63\\
\bottomrule
\end{tabular}}
\caption{Root Node Results, \texttt{A175} instances.}
\label{tbl:results:root:A175}
\end{table*}

\begin{table*}[h]
\resizebox{\textwidth}{!}{%
\begin{tabular}{ld{4.0}d{4.2}d{4.0}rrd{4.0}d{4.2}d{4.0}d{3.0}rrd{4.0}d{4.2}d{4.0}d{3.0}rr}
\toprule
&\multicolumn{5}{c}{Robust}&\multicolumn{6}{c}{Resource Robust}&\multicolumn{6}{c}{Non Robust}\\
\cmidrule(lr){2-6}\cmidrule(lr){7-12}\cmidrule(lr){13-18}
\multicolumn{1}{c}{Instance}&\multicolumn{1}{c}{Upper B.}&\multicolumn{1}{c}{Lower B.}&\multicolumn{1}{c}{CCs}&\multicolumn{1}{c}{Seconds}&\multicolumn{1}{c}{Nodes}&\multicolumn{1}{c}{Upper B.}&\multicolumn{1}{c}{Lower B.}&\multicolumn{1}{c}{CCs}&\multicolumn{1}{c}{$ng$-CCs}&\multicolumn{1}{c}{Seconds}&\multicolumn{1}{c}{Nodes}&\multicolumn{1}{c}{Upper B.}&\multicolumn{1}{c}{Lower B.}&\multicolumn{1}{c}{CCs}&\multicolumn{1}{c}{SCCs}&\multicolumn{1}{c}{Seconds}&\multicolumn{1}{c}{Nodes}\\
\midrule
\texttt{A175-n32-k5}&1173&=&117&1&1&1173&=&113&57&1&1&1173&=&102&65&3&1\\
\texttt{A175-n33-k5}&950&=&125&21&33&950&=&138&24&28&53&950&=&71&27&71&49\\
\texttt{A175-n33-k6}&1033&=&57&13&51&1033&=&27&11&1&5&1033&=&25&9&0&1\\
\texttt{A175-n34-k5}&1087&=&40&0&1&1087&=&41&26&1&1&1087&=&40&31&1&1\\
\texttt{A175-n36-k5}&1152&=&329&178&165&1152&=&94&33&56&57&1152&=&90&38&538&47\\
\texttt{A175-n37-k5}&895&=&130&51&34&895&=&120&31&51&22&895&=&110&34&332&25\\
\bottomrule
\end{tabular}}
\caption{Branch-Price-and-Cut Results, \texttt{A175} instances.}
\label{tbl:results:bpc:A175}
\end{table*}
\newpage

\begin{table*}[h]
\resizebox{\textwidth}{!}{%
\begin{tabular}{ld{4.2}d{3.0}d{2.2}d{1.2}d{4.2}d{3.0}d{3.0}d{2.2}d{1.2}d{4.2}d{3.0}d{3.0}d{2.2}d{1.2}}
\toprule
&\multicolumn{4}{c}{Robust}&\multicolumn{5}{c}{Resource Robust}&\multicolumn{5}{c}{Non Robust}\\
\cmidrule(lr){2-5}\cmidrule(lr){6-10}\cmidrule(lr){11-15}
\multicolumn{1}{c}{Instance}&\multicolumn{1}{c}{Lower Bound}&\multicolumn{1}{c}{CCs}&\multicolumn{1}{c}{Seconds}&\multicolumn{1}{c}{Gap \%}&\multicolumn{1}{c}{Lower Bound}&\multicolumn{1}{c}{CCs}&\multicolumn{1}{c}{$ng$-CCs}&\multicolumn{1}{c}{Seconds}&\multicolumn{1}{c}{Gap \%}&\multicolumn{1}{c}{Lower Bound}&\multicolumn{1}{c}{CCs}&\multicolumn{1}{c}{SCCs}&\multicolumn{1}{c}{Seconds}&\multicolumn{1}{c}{Gap \%}\\
\midrule
\texttt{A2-n32-k5}&1253&153&0.21&0.00&1253&153&15&0.26&0.00&1253&153&15&0.30&0.00\\
\texttt{A2-n33-k5}&1034&77&0.09&0.00&1034&65&18&0.10&0.00&1034&65&20&0.17&0.00\\
\texttt{A2-n33-k6}&1152.71&50&0.24&0.37&1153.78&48&4&0.18&0.28&1153.78&48&5&0.26&0.28\\
\texttt{A2-n34-k5}&1211&98&0.28&0.00&1211&94&19&0.25&0.00&1211&95&23&0.58&0.00\\
\texttt{A2-n36-k5}&1255.29&42&0.35&0.30&1255.44&41&51&0.47&0.28&1255.59&41&61&0.72&0.27\\
\texttt{A2-n37-k5}&934&133&0.40&0.00&934&64&36&0.46&0.00&934&64&36&0.65&0.00\\
\bottomrule
\end{tabular}}
\caption{Root Node Results, \texttt{A2} instances.}
\label{tbl:results:root:A2}
\end{table*}

\begin{table*}[h]
\resizebox{\textwidth}{!}{%
\begin{tabular}{ld{4.0}d{4.2}d{4.0}rrd{4.0}d{4.2}d{4.0}d{3.0}rrd{4.0}d{4.2}d{4.0}d{3.0}rr}
\toprule
&\multicolumn{5}{c}{Robust}&\multicolumn{6}{c}{Resource Robust}&\multicolumn{6}{c}{Non Robust}\\
\cmidrule(lr){2-6}\cmidrule(lr){7-12}\cmidrule(lr){13-18}
\multicolumn{1}{c}{Instance}&\multicolumn{1}{c}{Upper B.}&\multicolumn{1}{c}{Lower B.}&\multicolumn{1}{c}{CCs}&\multicolumn{1}{c}{Seconds}&\multicolumn{1}{c}{Nodes}&\multicolumn{1}{c}{Upper B.}&\multicolumn{1}{c}{Lower B.}&\multicolumn{1}{c}{CCs}&\multicolumn{1}{c}{$ng$-CCs}&\multicolumn{1}{c}{Seconds}&\multicolumn{1}{c}{Nodes}&\multicolumn{1}{c}{Upper B.}&\multicolumn{1}{c}{Lower B.}&\multicolumn{1}{c}{CCs}&\multicolumn{1}{c}{SCCs}&\multicolumn{1}{c}{Seconds}&\multicolumn{1}{c}{Nodes}\\
\midrule
\texttt{A2-n32-k5}&1253&=&153&0&1&1253&=&153&15&0&1&1253&=&153&15&0&1\\
\texttt{A2-n33-k5}&1034&=&77&0&1&1034&=&65&18&0&1&1034&=&65&20&0&1\\
\texttt{A2-n33-k6}&1157&=&77&6&35&1157&=&61&5&2&13&1157&=&71&6&3&13\\
\texttt{A2-n34-k5}&1211&=&98&0&1&1211&=&94&19&0&1&1211&=&95&23&1&1\\
\texttt{A2-n36-k5}&1259&=&124&13&23&1259&=&85&80&11&13&1259&=&116&110&19&11\\
\texttt{A2-n37-k5}&934&=&133&0&1&934&=&64&36&0&1&934&=&64&36&1&1\\
\bottomrule
\end{tabular}}
\caption{Branch-Price-and-Cut Results, \texttt{A2} instances.}
\label{tbl:results:bpc:A2}
\end{table*}
\newpage

\begin{table*}[h]
\resizebox{\textwidth}{!}{%
\begin{tabular}{ld{4.2}d{3.0}d{2.2}d{1.2}d{4.2}d{3.0}d{3.0}d{2.2}d{1.2}d{4.2}d{3.0}d{3.0}d{2.2}d{1.2}}
\toprule
&\multicolumn{4}{c}{Robust}&\multicolumn{5}{c}{Resource Robust}&\multicolumn{5}{c}{Non Robust}\\
\cmidrule(lr){2-5}\cmidrule(lr){6-10}\cmidrule(lr){11-15}
\multicolumn{1}{c}{Instance}&\multicolumn{1}{c}{Lower Bound}&\multicolumn{1}{c}{CCs}&\multicolumn{1}{c}{Seconds}&\multicolumn{1}{c}{Gap \%}&\multicolumn{1}{c}{Lower Bound}&\multicolumn{1}{c}{CCs}&\multicolumn{1}{c}{$ng$-CCs}&\multicolumn{1}{c}{Seconds}&\multicolumn{1}{c}{Gap \%}&\multicolumn{1}{c}{Lower Bound}&\multicolumn{1}{c}{CCs}&\multicolumn{1}{c}{SCCs}&\multicolumn{1}{c}{Seconds}&\multicolumn{1}{c}{Gap \%}\\
\midrule
\texttt{B-n31-k5}&672&44&1.62&0.00&672&28&39&3.51&0.00&672&29&48&3.81&0.00\\
\texttt{B-n34-k5}&784.94&121&5.21&0.39&785.56&98&207&11.15&0.31&785.98&99&263&224.24&0.26\\
\texttt{B-n35-k5}&955&113&3.79&0.00&955&73&168&10.65&0.00&955&73&168&90.97&0.00\\
\texttt{B-n38-k6}&804.09&38&3.91&0.11&804.38&31&40&9.00&0.08&804.33&27&43&10.83&0.08\\
\texttt{B-n39-k5}&549&75&30.59&0.00&549&67&231&114.78&0.00&549&76&285&295.59&0.00\\
\texttt{B-n41-k6}&828.16&66&3.66&0.10&828.54&46&21&4.70&0.06&828.73&47&21&54.73&0.03\\
\bottomrule
\end{tabular}}
\caption{Root Node Results, \texttt{B} instances.}
\label{tbl:results:root:B}
\end{table*}

\begin{table*}[h]
\resizebox{\textwidth}{!}{%
\begin{tabular}{ld{4.0}d{4.2}d{4.0}rrd{4.0}d{4.2}d{4.0}d{3.0}rrd{4.0}d{4.2}d{4.0}d{3.0}rr}
\toprule
&\multicolumn{5}{c}{Robust}&\multicolumn{6}{c}{Resource Robust}&\multicolumn{6}{c}{Non Robust}\\
\cmidrule(lr){2-6}\cmidrule(lr){7-12}\cmidrule(lr){13-18}
\multicolumn{1}{c}{Instance}&\multicolumn{1}{c}{Upper B.}&\multicolumn{1}{c}{Lower B.}&\multicolumn{1}{c}{CCs}&\multicolumn{1}{c}{Seconds}&\multicolumn{1}{c}{Nodes}&\multicolumn{1}{c}{Upper B.}&\multicolumn{1}{c}{Lower B.}&\multicolumn{1}{c}{CCs}&\multicolumn{1}{c}{$ng$-CCs}&\multicolumn{1}{c}{Seconds}&\multicolumn{1}{c}{Nodes}&\multicolumn{1}{c}{Upper B.}&\multicolumn{1}{c}{Lower B.}&\multicolumn{1}{c}{CCs}&\multicolumn{1}{c}{SCCs}&\multicolumn{1}{c}{Seconds}&\multicolumn{1}{c}{Nodes}\\
\midrule
\texttt{B-n31-k5}&672&=&44&2&1&672&=&28&39&4&1&672&=&29&48&4&1\\
\texttt{B-n34-k5}&788&=&160&441&133&788&=&101&213&143&15&788&786.15&101&263&\texttt{timeout}&12\\
\texttt{B-n35-k5}&955&=&113&4&1&955&=&73&168&11&1&955&=&73&168&91&1\\
\texttt{B-n38-k6}&805&=&46&41&17&805&=&36&40&32&9&805&=&36&46&175&17\\
\texttt{B-n39-k5}&549&=&75&31&1&549&=&67&231&115&1&549&=&76&285&296&1\\
\texttt{B-n41-k6}&829&=&75&35&9&829&=&56&24&68&17&829&=&53&23&771&15\\
\bottomrule
\end{tabular}}
\caption{Branch-Price-and-Cut Results, \texttt{B} instances.}
\label{tbl:results:bpc:B}
\end{table*}
\newpage

\begin{table*}[h]
\resizebox{\textwidth}{!}{%
\begin{tabular}{ld{4.2}d{3.0}d{2.2}d{1.2}d{4.2}d{3.0}d{3.0}d{2.2}d{1.2}d{4.2}d{3.0}d{3.0}d{2.2}d{1.2}}
\toprule
&\multicolumn{4}{c}{Robust}&\multicolumn{5}{c}{Resource Robust}&\multicolumn{5}{c}{Non Robust}\\
\cmidrule(lr){2-5}\cmidrule(lr){6-10}\cmidrule(lr){11-15}
\multicolumn{1}{c}{Instance}&\multicolumn{1}{c}{Lower Bound}&\multicolumn{1}{c}{CCs}&\multicolumn{1}{c}{Seconds}&\multicolumn{1}{c}{Gap \%}&\multicolumn{1}{c}{Lower Bound}&\multicolumn{1}{c}{CCs}&\multicolumn{1}{c}{$ng$-CCs}&\multicolumn{1}{c}{Seconds}&\multicolumn{1}{c}{Gap \%}&\multicolumn{1}{c}{Lower Bound}&\multicolumn{1}{c}{CCs}&\multicolumn{1}{c}{SCCs}&\multicolumn{1}{c}{Seconds}&\multicolumn{1}{c}{Gap \%}\\
\midrule
\texttt{B125-n31-k5}&783.50&32&0.96&0.45&783.48&25&28&1.61&0.45&783.80&30&47&10.11&0.41\\
\texttt{B125-n34-k5}&945&78&2.13&0.00&945&73&46&2.90&0.00&945&72&48&50.38&0.00\\
\texttt{B125-n35-k5}&1085.65&115&2.71&1.78&1088.26&60&288&6.11&1.54&1088.26&60&288&200.87&1.54\\
\texttt{B125-n38-k6}&872&138&1.48&0.00&872&119&111&3.32&0.00&872&119&111&4.27&0.00\\
\texttt{B125-n39-k5}&680.77&144&8.20&0.33&681.54&127&227&42.93&0.21&682.38&122&257&209.00&0.09\\
\texttt{B125-n41-k6}&971&43&0.90&0.00&971&37&31&1.38&0.00&971&37&32&1.32&0.00\\
\bottomrule
\end{tabular}}
\caption{Root Node Results, \texttt{B125} instances.}
\label{tbl:results:root:B125}
\end{table*}

\begin{table*}[h]
\resizebox{\textwidth}{!}{%
\begin{tabular}{ld{4.0}d{4.2}d{4.0}rrd{4.0}d{4.2}d{4.0}d{3.0}rrd{4.0}d{4.2}d{4.0}d{3.0}rr}
\toprule
&\multicolumn{5}{c}{Robust}&\multicolumn{6}{c}{Resource Robust}&\multicolumn{6}{c}{Non Robust}\\
\cmidrule(lr){2-6}\cmidrule(lr){7-12}\cmidrule(lr){13-18}
\multicolumn{1}{c}{Instance}&\multicolumn{1}{c}{Upper B.}&\multicolumn{1}{c}{Lower B.}&\multicolumn{1}{c}{CCs}&\multicolumn{1}{c}{Seconds}&\multicolumn{1}{c}{Nodes}&\multicolumn{1}{c}{Upper B.}&\multicolumn{1}{c}{Lower B.}&\multicolumn{1}{c}{CCs}&\multicolumn{1}{c}{$ng$-CCs}&\multicolumn{1}{c}{Seconds}&\multicolumn{1}{c}{Nodes}&\multicolumn{1}{c}{Upper B.}&\multicolumn{1}{c}{Lower B.}&\multicolumn{1}{c}{CCs}&\multicolumn{1}{c}{SCCs}&\multicolumn{1}{c}{Seconds}&\multicolumn{1}{c}{Nodes}\\
\midrule
\texttt{B125-n31-k5}&787&=&206&649&351&787&=&83&70&182&109&787&=&71&91&3174&84\\
\texttt{B125-n34-k5}&945&=&78&4&2&945&=&73&46&6&2&945&=&72&48&50&1\\
\texttt{B125-n35-k5}&\dsh&1099.02&342&\texttt{timeout}&467&1105&1103.04&165&358&\texttt{timeout}&296&1108&1090.12&84&300&\texttt{timeout}&5\\
\texttt{B125-n38-k6}&872&=&138&1&1&872&=&119&111&3&1&872&=&119&111&4&1\\
\texttt{B125-n39-k5}&683&=&162&230&57&683&=&136&230&307&21&683&=&125&258&2980&13\\
\texttt{B125-n41-k6}&971&=&43&1&1&971&=&37&31&1&1&971&=&37&32&1&1\\
\bottomrule
\end{tabular}}
\caption{Branch-Price-and-Cut Results, \texttt{B125} instances.}
\label{tbl:results:bpc:B125}
\end{table*}
\newpage

\begin{table*}[h]
\resizebox{\textwidth}{!}{%
\begin{tabular}{ld{4.2}d{3.0}d{2.2}d{1.2}d{4.2}d{3.0}d{3.0}d{2.2}d{1.2}d{4.2}d{3.0}d{3.0}d{2.2}d{1.2}}
\toprule
&\multicolumn{4}{c}{Robust}&\multicolumn{5}{c}{Resource Robust}&\multicolumn{5}{c}{Non Robust}\\
\cmidrule(lr){2-5}\cmidrule(lr){6-10}\cmidrule(lr){11-15}
\multicolumn{1}{c}{Instance}&\multicolumn{1}{c}{Lower Bound}&\multicolumn{1}{c}{CCs}&\multicolumn{1}{c}{Seconds}&\multicolumn{1}{c}{Gap \%}&\multicolumn{1}{c}{Lower Bound}&\multicolumn{1}{c}{CCs}&\multicolumn{1}{c}{$ng$-CCs}&\multicolumn{1}{c}{Seconds}&\multicolumn{1}{c}{Gap \%}&\multicolumn{1}{c}{Lower Bound}&\multicolumn{1}{c}{CCs}&\multicolumn{1}{c}{SCCs}&\multicolumn{1}{c}{Seconds}&\multicolumn{1}{c}{Gap \%}\\
\midrule
\texttt{B15-n31-k5}&887&50&0.61&0.00&887&39&40&0.97&0.00&887&30&47&1.75&0.00\\
\texttt{B15-n34-k5}&1066.12&42&1.77&0.74&1068.12&39&31&3.45&0.55&1071.80&33&31&80.49&0.21\\
\texttt{B15-n35-k5}&1268&99&0.94&0.00&1268&95&4&1.21&0.00&1268&96&4&2.12&0.00\\
\texttt{B15-n38-k6}&1039.79&56&1.50&0.21&1040.08&57&27&2.42&0.18&1040.94&58&31&22.29&0.10\\
\texttt{B15-n39-k5}&771.49&102&2.84&0.32&771.49&100&37&4.96&0.32&771.63&97&47&101.99&0.31\\
\texttt{B15-n41-k6}&1093.14&43&0.89&0.08&1093.14&38&6&1.13&0.08&1093.50&56&10&4.96&0.05\\
\bottomrule
\end{tabular}}
\caption{Root Node Results, \texttt{B15} instances.}
\label{tbl:results:root:B15}
\end{table*}

\begin{table*}[h]
\resizebox{\textwidth}{!}{%
\begin{tabular}{ld{4.0}d{4.2}d{4.0}rrd{4.0}d{4.2}d{4.0}d{3.0}rrd{4.0}d{4.2}d{4.0}d{3.0}rr}
\toprule
&\multicolumn{5}{c}{Robust}&\multicolumn{6}{c}{Resource Robust}&\multicolumn{6}{c}{Non Robust}\\
\cmidrule(lr){2-6}\cmidrule(lr){7-12}\cmidrule(lr){13-18}
\multicolumn{1}{c}{Instance}&\multicolumn{1}{c}{Upper B.}&\multicolumn{1}{c}{Lower B.}&\multicolumn{1}{c}{CCs}&\multicolumn{1}{c}{Seconds}&\multicolumn{1}{c}{Nodes}&\multicolumn{1}{c}{Upper B.}&\multicolumn{1}{c}{Lower B.}&\multicolumn{1}{c}{CCs}&\multicolumn{1}{c}{$ng$-CCs}&\multicolumn{1}{c}{Seconds}&\multicolumn{1}{c}{Nodes}&\multicolumn{1}{c}{Upper B.}&\multicolumn{1}{c}{Lower B.}&\multicolumn{1}{c}{CCs}&\multicolumn{1}{c}{SCCs}&\multicolumn{1}{c}{Seconds}&\multicolumn{1}{c}{Nodes}\\
\midrule
\texttt{B15-n31-k5}&887&=&50&1&1&887&=&39&40&1&1&887&=&30&47&2&1\\
\texttt{B15-n34-k5}&1074&=&129&650&253&1074&=&82&57&206&65&\dsh&1072.80&40&37&\texttt{timeout}&14\\
\texttt{B15-n35-k5}&1268&=&99&1&1&1268&=&95&4&1&1&1268&=&96&4&2&1\\
\texttt{B15-n38-k6}&1042&=&60&23&31&1042&=&65&28&28&19&1042&=&59&33&436&17\\
\texttt{B15-n39-k5}&774&=&129&85&27&774&=&118&45&72&25&\dsh&772.75&106&58&\texttt{timeout}&12\\
\texttt{B15-n41-k6}&1094&=&60&27&39&1094&=&60&15&34&43&1094&=&61&22&182&39\\
\bottomrule
\end{tabular}}
\caption{Branch-Price-and-Cut Results, \texttt{B15} instances.}
\label{tbl:results:bpc:B15}
\end{table*}
\newpage

\begin{table*}[h]
\resizebox{\textwidth}{!}{%
\begin{tabular}{ld{4.2}d{3.0}d{2.2}d{1.2}d{4.2}d{3.0}d{3.0}d{2.2}d{1.2}d{4.2}d{3.0}d{3.0}d{2.2}d{1.2}}
\toprule
&\multicolumn{4}{c}{Robust}&\multicolumn{5}{c}{Resource Robust}&\multicolumn{5}{c}{Non Robust}\\
\cmidrule(lr){2-5}\cmidrule(lr){6-10}\cmidrule(lr){11-15}
\multicolumn{1}{c}{Instance}&\multicolumn{1}{c}{Lower Bound}&\multicolumn{1}{c}{CCs}&\multicolumn{1}{c}{Seconds}&\multicolumn{1}{c}{Gap \%}&\multicolumn{1}{c}{Lower Bound}&\multicolumn{1}{c}{CCs}&\multicolumn{1}{c}{$ng$-CCs}&\multicolumn{1}{c}{Seconds}&\multicolumn{1}{c}{Gap \%}&\multicolumn{1}{c}{Lower Bound}&\multicolumn{1}{c}{CCs}&\multicolumn{1}{c}{SCCs}&\multicolumn{1}{c}{Seconds}&\multicolumn{1}{c}{Gap \%}\\
\midrule
\texttt{B175-n31-k5}&986&39&0.21&0.00&986&38&35&0.26&0.00&986&38&40&0.48&0.00\\
\texttt{B175-n34-k5}&1188&124&1.43&0.00&1188&109&207&3.32&0.00&1188&111&273&14.64&0.00\\
\texttt{B175-n35-k5}&1406&92&0.99&0.00&1406&92&15&0.95&0.00&1406&78&18&2.92&0.00\\
\texttt{B175-n38-k6}&1191.17&35&0.47&0.07&1192&34&17&0.57&0.00&1192&33&14&1.50&0.00\\
\texttt{B175-n39-k5}&819.08&16&0.66&0.36&819.12&14&6&1.19&0.35&819.12&14&6&2.02&0.35\\
\texttt{B175-n41-k6}&1229&33&0.44&0.00&1229&27&18&0.47&0.00&1229&27&22&1.93&0.00\\
\bottomrule
\end{tabular}}
\caption{Root Node Results, \texttt{B175} instances.}
\label{tbl:results:root:B175}
\end{table*}

\begin{table*}[h]
\resizebox{\textwidth}{!}{%
\begin{tabular}{ld{4.0}d{4.2}d{4.0}rrd{4.0}d{4.2}d{4.0}d{3.0}rrd{4.0}d{4.2}d{4.0}d{3.0}rr}
\toprule
&\multicolumn{5}{c}{Robust}&\multicolumn{6}{c}{Resource Robust}&\multicolumn{6}{c}{Non Robust}\\
\cmidrule(lr){2-6}\cmidrule(lr){7-12}\cmidrule(lr){13-18}
\multicolumn{1}{c}{Instance}&\multicolumn{1}{c}{Upper B.}&\multicolumn{1}{c}{Lower B.}&\multicolumn{1}{c}{CCs}&\multicolumn{1}{c}{Seconds}&\multicolumn{1}{c}{Nodes}&\multicolumn{1}{c}{Upper B.}&\multicolumn{1}{c}{Lower B.}&\multicolumn{1}{c}{CCs}&\multicolumn{1}{c}{$ng$-CCs}&\multicolumn{1}{c}{Seconds}&\multicolumn{1}{c}{Nodes}&\multicolumn{1}{c}{Upper B.}&\multicolumn{1}{c}{Lower B.}&\multicolumn{1}{c}{CCs}&\multicolumn{1}{c}{SCCs}&\multicolumn{1}{c}{Seconds}&\multicolumn{1}{c}{Nodes}\\
\midrule
\texttt{B175-n31-k5}&986&=&39&0&1&986&=&38&35&0&1&986&=&38&40&0&1\\
\texttt{B175-n34-k5}&1188&=&148&17&9&1188&=&109&207&3&1&1188&=&111&273&15&1\\
\texttt{B175-n35-k5}&1406&=&92&1&1&1406&=&92&15&1&1&1406&=&78&18&3&1\\
\texttt{B175-n38-k6}&1192&=&109&11&11&1192&=&34&17&1&1&1192&=&33&14&2&1\\
\texttt{B175-n39-k5}&822&=&34&36&41&822&=&50&18&34&17&822&=&18&11&228&49\\
\texttt{B175-n41-k6}&1229&=&33&0&1&1229&=&27&18&0&1&1229&=&27&22&2&1\\
\bottomrule
\end{tabular}}
\caption{Branch-Price-and-Cut Results, \texttt{B175} instances.}
\label{tbl:results:bpc:B175}
\end{table*}
\newpage

\begin{table*}[h]
\resizebox{\textwidth}{!}{%
\begin{tabular}{ld{4.2}d{3.0}d{2.2}d{1.2}d{4.2}d{3.0}d{3.0}d{2.2}d{1.2}d{4.2}d{3.0}d{3.0}d{2.2}d{1.2}}
\toprule
&\multicolumn{4}{c}{Robust}&\multicolumn{5}{c}{Resource Robust}&\multicolumn{5}{c}{Non Robust}\\
\cmidrule(lr){2-5}\cmidrule(lr){6-10}\cmidrule(lr){11-15}
\multicolumn{1}{c}{Instance}&\multicolumn{1}{c}{Lower Bound}&\multicolumn{1}{c}{CCs}&\multicolumn{1}{c}{Seconds}&\multicolumn{1}{c}{Gap \%}&\multicolumn{1}{c}{Lower Bound}&\multicolumn{1}{c}{CCs}&\multicolumn{1}{c}{$ng$-CCs}&\multicolumn{1}{c}{Seconds}&\multicolumn{1}{c}{Gap \%}&\multicolumn{1}{c}{Lower Bound}&\multicolumn{1}{c}{CCs}&\multicolumn{1}{c}{SCCs}&\multicolumn{1}{c}{Seconds}&\multicolumn{1}{c}{Gap \%}\\
\midrule
\texttt{B2-n31-k5}&1100&61&0.18&0.00&1099.00&62&12&0.18&0.09&1100&63&14&0.29&0.00\\
\texttt{B2-n34-k5}&1262&47&0.20&0.00&1262&47&88&0.76&0.00&1262&47&99&0.49&0.00\\
\texttt{B2-n35-k5}&1608.50&62&0.69&1.21&1613.38&48&31&0.71&0.91&1613.84&69&34&7.25&0.88\\
\texttt{B2-n38-k6}&1335.58&115&0.43&0.18&1336.00&127&42&0.82&0.15&1336.00&146&42&0.78&0.15\\
\texttt{B2-n39-k5}&952&62&0.70&0.00&952&48&32&1.08&0.00&952&41&41&4.73&0.00\\
\texttt{B2-n41-k6}&1475.31&85&0.48&0.11&1476.57&88&34&0.60&0.03&1476.85&72&31&4.85&0.01\\
\bottomrule
\end{tabular}}
\caption{Root Node Results, \texttt{B2} instances.}
\label{tbl:results:root:B2}
\end{table*}

\begin{table*}[h]
\resizebox{\textwidth}{!}{%
\begin{tabular}{ld{4.0}d{4.2}d{4.0}rrd{4.0}d{4.2}d{4.0}d{3.0}rrd{4.0}d{4.2}d{4.0}d{3.0}rr}
\toprule
&\multicolumn{5}{c}{Robust}&\multicolumn{6}{c}{Resource Robust}&\multicolumn{6}{c}{Non Robust}\\
\cmidrule(lr){2-6}\cmidrule(lr){7-12}\cmidrule(lr){13-18}
\multicolumn{1}{c}{Instance}&\multicolumn{1}{c}{Upper B.}&\multicolumn{1}{c}{Lower B.}&\multicolumn{1}{c}{CCs}&\multicolumn{1}{c}{Seconds}&\multicolumn{1}{c}{Nodes}&\multicolumn{1}{c}{Upper B.}&\multicolumn{1}{c}{Lower B.}&\multicolumn{1}{c}{CCs}&\multicolumn{1}{c}{$ng$-CCs}&\multicolumn{1}{c}{Seconds}&\multicolumn{1}{c}{Nodes}&\multicolumn{1}{c}{Upper B.}&\multicolumn{1}{c}{Lower B.}&\multicolumn{1}{c}{CCs}&\multicolumn{1}{c}{SCCs}&\multicolumn{1}{c}{Seconds}&\multicolumn{1}{c}{Nodes}\\
\midrule
\texttt{B2-n31-k5}&1100&=&61&0&1&1100&=&68&13&1&3&1100&=&63&14&0&1\\
\texttt{B2-n34-k5}&1262&=&47&0&1&1262&=&47&88&1&1&1262&=&47&99&0&1\\
\texttt{B2-n35-k5}&1628&=&370&800&767&1628&=&121&71&154&209&1631&1624.68&104&64&\texttt{timeout}&63\\
\texttt{B2-n38-k6}&1338&=&120&9&25&1338&=&127&42&7&13&1338&=&146&44&18&23\\
\texttt{B2-n39-k5}&952&=&62&1&1&952&=&48&32&1&1&952&=&41&41&5&1\\
\texttt{B2-n41-k6}&1477&=&146&19&40&1477&=&107&38&7&13&1477&=&73&33&39&15\\
\bottomrule
\end{tabular}}
\caption{Branch-Price-and-Cut Results, \texttt{B2} instances.}
\label{tbl:results:bpc:B2}
\end{table*}
\newpage

\begin{table*}[h]
\resizebox{\textwidth}{!}{%
\begin{tabular}{ld{4.2}d{3.0}d{2.2}d{1.2}d{4.2}d{3.0}d{3.0}d{2.2}d{1.2}d{4.2}d{3.0}d{3.0}d{2.2}d{1.2}}
\toprule
&\multicolumn{4}{c}{Robust}&\multicolumn{5}{c}{Resource Robust}&\multicolumn{5}{c}{Non Robust}\\
\cmidrule(lr){2-5}\cmidrule(lr){6-10}\cmidrule(lr){11-15}
\multicolumn{1}{c}{Instance}&\multicolumn{1}{c}{Lower Bound}&\multicolumn{1}{c}{CCs}&\multicolumn{1}{c}{Seconds}&\multicolumn{1}{c}{Gap \%}&\multicolumn{1}{c}{Lower Bound}&\multicolumn{1}{c}{CCs}&\multicolumn{1}{c}{$ng$-CCs}&\multicolumn{1}{c}{Seconds}&\multicolumn{1}{c}{Gap \%}&\multicolumn{1}{c}{Lower Bound}&\multicolumn{1}{c}{CCs}&\multicolumn{1}{c}{SCCs}&\multicolumn{1}{c}{Seconds}&\multicolumn{1}{c}{Gap \%}\\
\midrule
\texttt{P-n16-k8}&448.00&32&0.02&0.45&450&32&1&0.03&0.00&450&32&1&0.04&0.00\\
\texttt{P-n19-k2}&211.67&6&1.84&0.16&212&8&6&5.16&0.00&212&6&3&12.76&0.00\\
\texttt{P-n20-k2}&215.50&7&2.05&0.23&215.67&7&3&3.74&0.15&215.67&7&3&9.36&0.15\\
\texttt{P-n21-k2}&211&0&0.92&0.00&211&0&0&0.96&0.00&211&0&0&1.16&0.00\\
\texttt{P-n22-k2}&215.50&0&1.23&0.23&215.50&0&0&1.22&0.23&215.50&0&0&1.61&0.23\\
\texttt{P-n22-k8}&590&46&0.05&0.00&590&46&3&0.05&0.00&590&46&3&0.04&0.00\\
\bottomrule
\end{tabular}}
\caption{Root Node Results, \texttt{P} instances.}
\label{tbl:results:root:P}
\end{table*}

\begin{table*}[h]
\resizebox{\textwidth}{!}{%
\begin{tabular}{ld{4.0}d{4.2}d{4.0}rrd{4.0}d{4.2}d{4.0}d{3.0}rrd{4.0}d{4.2}d{4.0}d{3.0}rr}
\toprule
&\multicolumn{5}{c}{Robust}&\multicolumn{6}{c}{Resource Robust}&\multicolumn{6}{c}{Non Robust}\\
\cmidrule(lr){2-6}\cmidrule(lr){7-12}\cmidrule(lr){13-18}
\multicolumn{1}{c}{Instance}&\multicolumn{1}{c}{Upper B.}&\multicolumn{1}{c}{Lower B.}&\multicolumn{1}{c}{CCs}&\multicolumn{1}{c}{Seconds}&\multicolumn{1}{c}{Nodes}&\multicolumn{1}{c}{Upper B.}&\multicolumn{1}{c}{Lower B.}&\multicolumn{1}{c}{CCs}&\multicolumn{1}{c}{$ng$-CCs}&\multicolumn{1}{c}{Seconds}&\multicolumn{1}{c}{Nodes}&\multicolumn{1}{c}{Upper B.}&\multicolumn{1}{c}{Lower B.}&\multicolumn{1}{c}{CCs}&\multicolumn{1}{c}{SCCs}&\multicolumn{1}{c}{Seconds}&\multicolumn{1}{c}{Nodes}\\
\midrule
\texttt{P-n16-k8}&450&=&37&0&9&450&=&32&1&0&1&450&=&32&1&0&1\\
\texttt{P-n19-k2}&212&=&6&11&9&212&=&8&6&5&1&212&=&6&3&13&1\\
\texttt{P-n20-k2}&216&=&7&26&17&216&=&7&4&45&21&216&=&7&3&128&17\\
\texttt{P-n21-k2}&211&=&0&1&1&211&=&0&0&1&1&211&=&0&0&1&1\\
\texttt{P-n22-k2}&216&=&4&17&5&216&=&4&0&16&5&216&=&4&1&33&5\\
\texttt{P-n22-k8}&590&=&46&0&1&590&=&46&3&0&1&590&=&46&3&0&1\\
\bottomrule
\end{tabular}}
\caption{Branch-Price-and-Cut Results, \texttt{P} instances.}
\label{tbl:results:bpc:P}
\end{table*}
\newpage

\begin{table*}[h]
\resizebox{\textwidth}{!}{%
\begin{tabular}{ld{4.2}d{3.0}d{2.2}d{1.2}d{4.2}d{3.0}d{3.0}d{2.2}d{1.2}d{4.2}d{3.0}d{3.0}d{2.2}d{1.2}}
\toprule
&\multicolumn{4}{c}{Robust}&\multicolumn{5}{c}{Resource Robust}&\multicolumn{5}{c}{Non Robust}\\
\cmidrule(lr){2-5}\cmidrule(lr){6-10}\cmidrule(lr){11-15}
\multicolumn{1}{c}{Instance}&\multicolumn{1}{c}{Lower Bound}&\multicolumn{1}{c}{CCs}&\multicolumn{1}{c}{Seconds}&\multicolumn{1}{c}{Gap \%}&\multicolumn{1}{c}{Lower Bound}&\multicolumn{1}{c}{CCs}&\multicolumn{1}{c}{$ng$-CCs}&\multicolumn{1}{c}{Seconds}&\multicolumn{1}{c}{Gap \%}&\multicolumn{1}{c}{Lower Bound}&\multicolumn{1}{c}{CCs}&\multicolumn{1}{c}{SCCs}&\multicolumn{1}{c}{Seconds}&\multicolumn{1}{c}{Gap \%}\\
\midrule
\texttt{P125-n16-k8}&524&14&0.01&0.00&524&14&0&0.01&0.00&524&14&0&0.01&0.00\\
\texttt{P125-n19-k2}&246&31&0.50&0.00&246&28&69&1.03&0.00&246&28&69&1.07&0.00\\
\texttt{P125-n20-k2}&250.50&30&0.61&0.20&251&26&83&1.57&0.00&251&26&83&3.19&0.00\\
\texttt{P125-n21-k2}&249&40&1.00&0.00&249&22&94&1.46&0.00&249&22&95&2.40&0.00\\
\texttt{P125-n22-k2}&255.00&41&2.23&0.39&255.00&26&118&4.03&0.39&255.00&26&121&5.57&0.39\\
\texttt{P125-n22-k8}&715&8&0.03&0.00&715&8&3&0.04&0.00&715&8&3&0.04&0.00\\
\bottomrule
\end{tabular}}
\caption{Root Node Results, \texttt{P125} instances.}
\label{tbl:results:root:P125}
\end{table*}

\begin{table*}[h]
\resizebox{\textwidth}{!}{%
\begin{tabular}{ld{4.0}d{4.2}d{4.0}rrd{4.0}d{4.2}d{4.0}d{3.0}rrd{4.0}d{4.2}d{4.0}d{3.0}rr}
\toprule
&\multicolumn{5}{c}{Robust}&\multicolumn{6}{c}{Resource Robust}&\multicolumn{6}{c}{Non Robust}\\
\cmidrule(lr){2-6}\cmidrule(lr){7-12}\cmidrule(lr){13-18}
\multicolumn{1}{c}{Instance}&\multicolumn{1}{c}{Upper B.}&\multicolumn{1}{c}{Lower B.}&\multicolumn{1}{c}{CCs}&\multicolumn{1}{c}{Seconds}&\multicolumn{1}{c}{Nodes}&\multicolumn{1}{c}{Upper B.}&\multicolumn{1}{c}{Lower B.}&\multicolumn{1}{c}{CCs}&\multicolumn{1}{c}{$ng$-CCs}&\multicolumn{1}{c}{Seconds}&\multicolumn{1}{c}{Nodes}&\multicolumn{1}{c}{Upper B.}&\multicolumn{1}{c}{Lower B.}&\multicolumn{1}{c}{CCs}&\multicolumn{1}{c}{SCCs}&\multicolumn{1}{c}{Seconds}&\multicolumn{1}{c}{Nodes}\\
\midrule
\texttt{P125-n16-k8}&524&=&14&0&1&524&=&14&0&0&1&524&=&14&0&0&1\\
\texttt{P125-n19-k2}&246&=&31&1&1&246&=&28&69&1&1&246&=&28&69&1&1\\
\texttt{P125-n20-k2}&251&=&30&3&5&251&=&26&83&2&1&251&=&26&83&3&1\\
\texttt{P125-n21-k2}&249&=&40&1&1&249&=&22&94&1&1&249&=&22&95&2&1\\
\texttt{P125-n22-k2}&256&=&50&21&13&256&=&32&120&76&21&256&=&29&122&97&9\\
\texttt{P125-n22-k8}&715&=&8&0&1&715&=&8&3&0&1&715&=&8&3&0&1\\
\bottomrule
\end{tabular}}
\caption{Branch-Price-and-Cut Results, \texttt{P125} instances.}
\label{tbl:results:bpc:P125}
\end{table*}
\newpage

\begin{table*}[h]
\resizebox{\textwidth}{!}{%
\begin{tabular}{ld{4.2}d{3.0}d{2.2}d{1.2}d{4.2}d{3.0}d{3.0}d{2.2}d{1.2}d{4.2}d{3.0}d{3.0}d{2.2}d{1.2}}
\toprule
&\multicolumn{4}{c}{Robust}&\multicolumn{5}{c}{Resource Robust}&\multicolumn{5}{c}{Non Robust}\\
\cmidrule(lr){2-5}\cmidrule(lr){6-10}\cmidrule(lr){11-15}
\multicolumn{1}{c}{Instance}&\multicolumn{1}{c}{Lower Bound}&\multicolumn{1}{c}{CCs}&\multicolumn{1}{c}{Seconds}&\multicolumn{1}{c}{Gap \%}&\multicolumn{1}{c}{Lower Bound}&\multicolumn{1}{c}{CCs}&\multicolumn{1}{c}{$ng$-CCs}&\multicolumn{1}{c}{Seconds}&\multicolumn{1}{c}{Gap \%}&\multicolumn{1}{c}{Lower Bound}&\multicolumn{1}{c}{CCs}&\multicolumn{1}{c}{SCCs}&\multicolumn{1}{c}{Seconds}&\multicolumn{1}{c}{Gap \%}\\
\midrule
\texttt{P15-n16-k8}&582&0&0.01&0.00&582&0&0&0.01&0.00&582&0&0&0.00&0.00\\
\texttt{P15-n19-k2}&258&14&0.19&0.00&258&14&8&0.27&0.00&258&14&8&0.77&0.00\\
\texttt{P15-n20-k2}&261&0&0.17&0.00&261&0&0&0.20&0.00&261&0&0&0.20&0.00\\
\texttt{P15-n21-k2}&261&91&0.93&0.00&261&95&1&1.78&0.00&261&96&1&5.74&0.00\\
\texttt{P15-n22-k2}&265&12&0.56&0.00&265&8&5&1.00&0.00&265&8&5&1.99&0.00\\
\texttt{P15-n22-k8}&761&0&0.02&0.00&761&0&0&0.01&0.00&761&0&0&0.01&0.00\\
\bottomrule
\end{tabular}}
\caption{Root Node Results, \texttt{P15} instances.}
\label{tbl:results:root:P15}
\end{table*}

\begin{table*}[h]
\resizebox{\textwidth}{!}{%
\begin{tabular}{ld{4.0}d{4.2}d{4.0}rrd{4.0}d{4.2}d{4.0}d{3.0}rrd{4.0}d{4.2}d{4.0}d{3.0}rr}
\toprule
&\multicolumn{5}{c}{Robust}&\multicolumn{6}{c}{Resource Robust}&\multicolumn{6}{c}{Non Robust}\\
\cmidrule(lr){2-6}\cmidrule(lr){7-12}\cmidrule(lr){13-18}
\multicolumn{1}{c}{Instance}&\multicolumn{1}{c}{Upper B.}&\multicolumn{1}{c}{Lower B.}&\multicolumn{1}{c}{CCs}&\multicolumn{1}{c}{Seconds}&\multicolumn{1}{c}{Nodes}&\multicolumn{1}{c}{Upper B.}&\multicolumn{1}{c}{Lower B.}&\multicolumn{1}{c}{CCs}&\multicolumn{1}{c}{$ng$-CCs}&\multicolumn{1}{c}{Seconds}&\multicolumn{1}{c}{Nodes}&\multicolumn{1}{c}{Upper B.}&\multicolumn{1}{c}{Lower B.}&\multicolumn{1}{c}{CCs}&\multicolumn{1}{c}{SCCs}&\multicolumn{1}{c}{Seconds}&\multicolumn{1}{c}{Nodes}\\
\midrule
\texttt{P15-n16-k8}&582&=&0&0&1&582&=&0&0&0&1&582&=&0&0&0&1\\
\texttt{P15-n19-k2}&258&=&14&0&1&258&=&14&8&0&1&258&=&14&8&1&1\\
\texttt{P15-n20-k2}&261&=&0&0&1&261&=&0&0&0&1&261&=&0&0&0&1\\
\texttt{P15-n21-k2}&261&=&91&6&5&261&=&95&1&8&5&261&=&96&1&12&3\\
\texttt{P15-n22-k2}&265&=&12&1&1&265&=&8&5&1&1&265&=&8&5&2&1\\
\texttt{P15-n22-k8}&761&=&0&0&1&761&=&0&0&0&1&761&=&0&0&0&1\\
\bottomrule
\end{tabular}}
\caption{Branch-Price-and-Cut Results, \texttt{P15} instances.}
\label{tbl:results:bpc:P15}
\end{table*}
\newpage

\begin{table*}[h]
\resizebox{\textwidth}{!}{%
\begin{tabular}{ld{4.2}d{3.0}d{2.2}d{1.2}d{4.2}d{3.0}d{3.0}d{2.2}d{1.2}d{4.2}d{3.0}d{3.0}d{2.2}d{1.2}}
\toprule
&\multicolumn{4}{c}{Robust}&\multicolumn{5}{c}{Resource Robust}&\multicolumn{5}{c}{Non Robust}\\
\cmidrule(lr){2-5}\cmidrule(lr){6-10}\cmidrule(lr){11-15}
\multicolumn{1}{c}{Instance}&\multicolumn{1}{c}{Lower Bound}&\multicolumn{1}{c}{CCs}&\multicolumn{1}{c}{Seconds}&\multicolumn{1}{c}{Gap \%}&\multicolumn{1}{c}{Lower Bound}&\multicolumn{1}{c}{CCs}&\multicolumn{1}{c}{$ng$-CCs}&\multicolumn{1}{c}{Seconds}&\multicolumn{1}{c}{Gap \%}&\multicolumn{1}{c}{Lower Bound}&\multicolumn{1}{c}{CCs}&\multicolumn{1}{c}{SCCs}&\multicolumn{1}{c}{Seconds}&\multicolumn{1}{c}{Gap \%}\\
\midrule
\texttt{P175-n16-k8}&638&2&0.01&0.00&638&2&0&0.01&0.00&638&2&0&0.01&0.00\\
\texttt{P175-n19-k2}&288&41&0.21&0.00&288&41&36&0.28&0.00&288&41&36&0.34&0.00\\
\texttt{P175-n20-k2}&297.00&31&0.32&0.34&297.00&30&60&0.49&0.34&297.00&30&61&0.68&0.34\\
\texttt{P175-n21-k2}&295.14&35&0.53&0.97&295.14&31&34&0.87&0.97&295.14&31&44&1.55&0.97\\
\texttt{P175-n22-k2}&300&35&0.34&0.00&300&21&91&0.56&0.00&300&21&91&0.61&0.00\\
\texttt{P175-n22-k8}&871&0&0.01&0.00&871&0&0&0.01&0.00&871&0&0&0.01&0.00\\
\bottomrule
\end{tabular}}
\caption{Root Node Results, \texttt{P175} instances.}
\label{tbl:results:root:P175}
\end{table*}

\begin{table*}[h]
\resizebox{\textwidth}{!}{%
\begin{tabular}{ld{4.0}d{4.2}d{4.0}rrd{4.0}d{4.2}d{4.0}d{3.0}rrd{4.0}d{4.2}d{4.0}d{3.0}rr}
\toprule
&\multicolumn{5}{c}{Robust}&\multicolumn{6}{c}{Resource Robust}&\multicolumn{6}{c}{Non Robust}\\
\cmidrule(lr){2-6}\cmidrule(lr){7-12}\cmidrule(lr){13-18}
\multicolumn{1}{c}{Instance}&\multicolumn{1}{c}{Upper B.}&\multicolumn{1}{c}{Lower B.}&\multicolumn{1}{c}{CCs}&\multicolumn{1}{c}{Seconds}&\multicolumn{1}{c}{Nodes}&\multicolumn{1}{c}{Upper B.}&\multicolumn{1}{c}{Lower B.}&\multicolumn{1}{c}{CCs}&\multicolumn{1}{c}{$ng$-CCs}&\multicolumn{1}{c}{Seconds}&\multicolumn{1}{c}{Nodes}&\multicolumn{1}{c}{Upper B.}&\multicolumn{1}{c}{Lower B.}&\multicolumn{1}{c}{CCs}&\multicolumn{1}{c}{SCCs}&\multicolumn{1}{c}{Seconds}&\multicolumn{1}{c}{Nodes}\\
\midrule
\texttt{P175-n16-k8}&638&=&2&0&1&638&=&2&0&0&1&638&=&2&0&0&1\\
\texttt{P175-n19-k2}&288&=&41&0&1&288&=&41&36&0&1&288&=&41&36&0&1\\
\texttt{P175-n20-k2}&298&=&32&2&5&298&=&31&61&3&5&298&=&31&61&4&5\\
\texttt{P175-n21-k2}&298&=&51&24&67&298&=&46&45&64&101&298&=&45&53&156&93\\
\texttt{P175-n22-k2}&300&=&35&0&1&300&=&21&91&1&1&300&=&21&91&1&1\\
\texttt{P175-n22-k8}&871&=&0&0&1&871&=&0&0&0&1&871&=&0&0&0&1\\
\bottomrule
\end{tabular}}
\caption{Branch-Price-and-Cut Results, \texttt{P175} instances.}
\label{tbl:results:bpc:P175}
\end{table*}
\newpage

\begin{table*}[h]
\resizebox{\textwidth}{!}{%
\begin{tabular}{ld{4.2}d{3.0}d{2.2}d{1.2}d{4.2}d{3.0}d{3.0}d{2.2}d{1.2}d{4.2}d{3.0}d{3.0}d{2.2}d{1.2}}
\toprule
&\multicolumn{4}{c}{Robust}&\multicolumn{5}{c}{Resource Robust}&\multicolumn{5}{c}{Non Robust}\\
\cmidrule(lr){2-5}\cmidrule(lr){6-10}\cmidrule(lr){11-15}
\multicolumn{1}{c}{Instance}&\multicolumn{1}{c}{Lower Bound}&\multicolumn{1}{c}{CCs}&\multicolumn{1}{c}{Seconds}&\multicolumn{1}{c}{Gap \%}&\multicolumn{1}{c}{Lower Bound}&\multicolumn{1}{c}{CCs}&\multicolumn{1}{c}{$ng$-CCs}&\multicolumn{1}{c}{Seconds}&\multicolumn{1}{c}{Gap \%}&\multicolumn{1}{c}{Lower Bound}&\multicolumn{1}{c}{CCs}&\multicolumn{1}{c}{SCCs}&\multicolumn{1}{c}{Seconds}&\multicolumn{1}{c}{Gap \%}\\
\midrule
\texttt{P2-n16-k8}&676&0&0.01&0.00&676&0&0&0.01&0.00&676&0&0&0.01&0.00\\
\texttt{P2-n19-k2}&308.00&2&0.10&0.32&309&2&3&0.14&0.00&309&2&3&0.20&0.00\\
\texttt{P2-n20-k2}&311.00&10&0.21&0.32&311.00&8&1&0.19&0.32&311.00&8&1&0.31&0.32\\
\texttt{P2-n21-k2}&304&52&0.20&0.00&304&52&1&0.20&0.00&304&52&1&0.30&0.00\\
\texttt{P2-n22-k2}&310.50&0&0.11&0.16&310.50&0&0&0.09&0.16&310.50&0&0&0.15&0.16\\
\texttt{P2-n22-k8}&924&0&0.01&0.00&924&0&0&0.01&0.00&924&0&0&0.01&0.00\\
\bottomrule
\end{tabular}}
\caption{Root Node Results, \texttt{P2} instances.}
\label{tbl:results:root:P2}
\end{table*}

\begin{table*}[h]
\resizebox{\textwidth}{!}{%
\begin{tabular}{ld{4.0}d{4.2}d{4.0}rrd{4.0}d{4.2}d{4.0}d{3.0}rrd{4.0}d{4.2}d{4.0}d{3.0}rr}
\toprule
&\multicolumn{5}{c}{Robust}&\multicolumn{6}{c}{Resource Robust}&\multicolumn{6}{c}{Non Robust}\\
\cmidrule(lr){2-6}\cmidrule(lr){7-12}\cmidrule(lr){13-18}
\multicolumn{1}{c}{Instance}&\multicolumn{1}{c}{Upper B.}&\multicolumn{1}{c}{Lower B.}&\multicolumn{1}{c}{CCs}&\multicolumn{1}{c}{Seconds}&\multicolumn{1}{c}{Nodes}&\multicolumn{1}{c}{Upper B.}&\multicolumn{1}{c}{Lower B.}&\multicolumn{1}{c}{CCs}&\multicolumn{1}{c}{$ng$-CCs}&\multicolumn{1}{c}{Seconds}&\multicolumn{1}{c}{Nodes}&\multicolumn{1}{c}{Upper B.}&\multicolumn{1}{c}{Lower B.}&\multicolumn{1}{c}{CCs}&\multicolumn{1}{c}{SCCs}&\multicolumn{1}{c}{Seconds}&\multicolumn{1}{c}{Nodes}\\
\midrule
\texttt{P2-n16-k8}&676&=&0&0&1&676&=&0&0&0&1&676&=&0&0&0&1\\
\texttt{P2-n19-k2}&309&=&14&3&21&309&=&2&3&0&1&309&=&2&3&0&1\\
\texttt{P2-n20-k2}&312&=&10&1&5&312&=&8&2&1&5&312&=&8&3&1&3\\
\texttt{P2-n21-k2}&304&=&52&0&1&304&=&52&1&0&1&304&=&52&1&0&1\\
\texttt{P2-n22-k2}&311&=&0&1&7&311&=&0&0&1&7&311&=&0&0&1&7\\
\texttt{P2-n22-k8}&924&=&0&0&1&924&=&0&0&0&1&924&=&0&0&0&1\\
\bottomrule
\end{tabular}}
\caption{Branch-Price-and-Cut Results, \texttt{P2} instances.}
\label{tbl:results:bpc:P2}
\end{table*}
\newpage